\documentclass[12pt,dvipsnames,reqno]{amsart}
\usepackage{amsmath,amssymb,amsthm,amstext}
\usepackage{enumerate}
\usepackage{graphicx}
\usepackage{amsfonts}
\usepackage{hyperref}
\usepackage{epsfig}
\usepackage{epstopdf}
\usepackage{caption}
\usepackage{bigints}

\usepackage{scalerel}[2016-12-29]

\def\scaleint#1{\vcenter{\hbox{\scaleto[3ex]{\displaystyle\int}{#1}}}}

\usepackage{tikz}
\usepackage{ifthen}
\usetikzlibrary{patterns}
\usepackage{functan}
\usepackage[utf8]{inputenc}
\usepackage[T1]{fontenc}
\usepackage{dutchcal}
\usepackage{mathtools}

\usepackage{color}
\usepackage{xcolor}

\usepackage[margin=1 in]{geometry}
\usepackage{setspace}
\usepackage{subcaption}
\usepackage{comment}
\usepackage{float}
\usepackage{chngcntr}
\usepackage{mathrsfs}
\numberwithin{equation}{section}
\newtheorem{thm}{Theorem}[section]

\newtheorem{lem}[thm]{Lemma}

\newtheorem{prop}[thm]{Proposition}
\theoremstyle{remark}

\newcommand{\R}{\mathbb{R}}

\def\qq#1{\qquad \mbox{#1}\quad}
\def\q#1{\quad \mbox{#1}\ }
\newcommand{\al}{\alpha}
\newcommand{\be}{\beta}
\newcommand{\De}{\Delta}
\newcommand{\de}{\delta}

\newcommand{\e}{\varepsilon}
\newcommand{\eps}{\e }
\newcommand{\g}{\gamma}
\newcommand{\Ga}{\Gamma}
\newcommand{\la}{\lambda}

\newcommand{\Om}{\Omega}
\newcommand{\Omb}{\overline{\Om}}
\newcommand{\p}{\partial}
\newcommand{\s}{\sigma}
\newcommand{\te}{\theta}

\newcommand{\vf}{\varphi}

\newcommand{\z}{\zeta}

\title[
Elliptic systems with superlinear nonlinearities on the boundary]{
Positive solutions of elliptic systems with superlinear nonlinearities on the boundary} 
{\small
\author[S.~Bandyopadhyay]{Shalmali Bandyopadhyay}
\address{S. Bandyopadhyay \newline The University of Tennessee at Martin, Martin, TN, USA}
\email{sbandyo5@utm.edu}

\author[M.~Chhetri]{Maya Chhetri}
\address{M. Chhetri \newline 
Department of Mathematics and Statistics, UNC Greensboro, Greensboro, NC, USA}
\email{m\_chhetr@uncg.edu}
\thanks{Corresponding author: M.~Chhetri (m\_chhetr@uncg.edu)}

\author[B.~B.~Delgado]{Briceyda B. Delgado}
\address{B.~.B.~Delgado \newline 
INFOTEC, Centro de Investigación e Innovación en Tecnologías \\de la Información y Comunicación, Aguascalientes, Mexico}
\email{briceyda.delgado@infotec.edu.mx}

\author[N.~Mavinga]{Nsoki Mavinga}
 \address{N.~Mavinga \newline 
 Department of Mathematics and Statistics, Swarthmore College, Swarthmore, PA, USA}
 \email{nmaving1@swarthmore.edu}

 \author[R.~Pardo]{Rosa Pardo}
 \address{R.~Pardo \newline 
 Departamento de An\'alisis Matem\'atico y Matem\'atica Aplicada\\ 
 Universidad Complutense de Madrid, 28040--Madrid, Spain}
 \email{rpardo@ucm.es}
}

\begin{document}

\begin{abstract}

We consider  elliptic  systems with superlinear and subcritical boundary conditions and a bifurcation parameter as a multiplicative factor. By combining the rescaling method with degree theory and elliptic regularity theory, we prove the existence of a connected branch of positive  weak solutions that bifurcates from infinity as the parameter approaches zero. Furthermore, under additional conditions on the nonlinearities near zero, we obtain a global connected branch of positive solutions bifurcating from zero, which possesses a unique bifurcation point from infinity  when the parameter is zero. Finally, we analyze the behavior of this branch and discuss the number of positive weak solutions with respect to the parameter using bifurcation theory, degree theory,  and sub- and super-solution methods. 
\end{abstract}
\maketitle

\section{Introduction}
Consider the following linear system with nonlinear boundary conditions of the form
\begin{equation}
\label{pde}
	\left.
	\begin{array}{rclll}
-\De u_1+u_1 & = 0 \quad \mbox{in}\quad \Omega\,,&
&\frac{\partial u_1}{\partial \eta}&=\lambda f_1(u_2)\quad \mbox {on}\quad \partial\Omega\,;\\
-\Delta u_2 +u_2 &=  0 \quad \mbox{in}\quad \Omega\,, &
&\frac{\partial u_2}{\partial \eta}&=\lambda f_2(u_1)\quad \mbox {on}\quad \partial\Omega\,;
\end{array}
	\right\} 
\end{equation}
where $\Omega\subset \mathbb{R}^N$ ($N > 2$) is a bounded domain with a $C^{2,\g}$  boundary $\p\Om$ (with $0<\g<1$), $\partial/\partial \eta:=\eta(x) \cdot \nabla$ denotes the outer normal derivative on $\partial\Omega$, and $\lambda>0$ is a bifurcation parameter.  The nonlinear reaction terms $f_i\colon [0,\infty)\rightarrow [0,\infty)$ are  locally Hölder continuous functions, that is, given any positive constant $M_0$, there exists $L$ such that   $|f_i(s)-f_i(t)|\le L|s-t|^{\g}$  for all $s,t\in [-M_0,M_0]$ and $0<\g\le 1$, for $i=1, 2$.
\newline
\par The goal of this work is to investigate the number of positive solutions of \eqref{pde} with respect to the bifurcation parameter $\lambda$. It is well known that the asymptotic behaviors of the nonlinearities near the origin and/or at infinity influence the number of solutions with respect to $\lambda$. Here, the focus will be on the case when $f_1$ and $f_2$ are superlinear and subcritical at infinity, that is, there exist constants $b_i>0$ ($i=1, 2$) such that 
\begin{equation}
   \label{H_inf}
   \tag{${\rm H}_{\infty}$}
\begin{cases}
  \lim_{s\rightarrow \infty}\dfrac{f_1(s)}{s^{p_2}}=b_2, \quad \lim_{s\rightarrow \infty}\dfrac{f_2(s)}{s^{p_1}}=b_1,\\
   1<p_1\,, p_2 \le \frac{N}{N-2} \quad \mbox{but not both  equal to}\quad  \frac{N}{N-2}\,.  
\end{cases}
\end{equation}

\medskip 

\par We are concerned with the existence, multiplicity and bifurcation results for positive weak solutions of \eqref{pde} in $H^1(\Omega) \times H^1(\Omega)$ with respect to the parameter $\lambda$. However, we show that if $f_1$ and $f_2$ satisfy the hypothesis \eqref{H_inf} and $(u_1,u_2)$ is a weak solution, then $u_1,\,u_2 \in C^{2,\alpha}(\Omega)\cap C^{1,\alpha}(\overline{\Omega})$ for some $\alpha \in (0, 1)$, see Theorem~\ref{th:bootstrap}. 
Thus, we can define our solution set by
\begin{equation}
\label{solution:set}
\Sigma:=\{\big(\lambda,(u_1,u_2)\big)\in  (0,\infty)\times \big(C(\overline{\Omega})\big)^2\colon (\lambda, (u_1,u_2)) \mbox{ is a weak solution to }\eqref{pde}\}\,,
\end{equation}
and look for solutions $\big(\lambda,(u_{1},u_{2})\big) \in \Sigma.$


\medskip

Throughout the paper,  the norm on the underlying function space $(C(\Omb))^2 $ is given by
\begin{equation}\label{norm}
\|(u_1,u_2)\|_{C(\overline{\Omega})^2}:=\|u_1\|_{C(\Omb)}+\|u_2\|_{C(\Omb)}\,.    
\end{equation} 
Furthermore,  $\mu_1>0$ will denote the first  eigenvalue of the Steklov eigenvalue problem
\begin{equation}\label{steklov}
\left\{
\begin{array}{rcll}
-\Delta \psi +\psi &=&  0 \quad &\mbox{in}\quad \Om\,,\\
\frac{\p \psi}{\p \eta} &=& \mu\psi\quad &\mbox{on}\quad \p \Om\,,
\end{array}
\right. 
\end{equation}
with corresponding eigenfunction $\varphi_1>0$ on $\overline{\Om}$ and $\|\varphi_1\|_{L^{\infty}(\Om)}=1$, see \cite[Prop.~2.3 and Rem.~2.4]{Mav_2012}.

\medskip

\par Our first result shows that by assuming only the superlinear subcritical behaviors at infinity by the nonlinearities $f_1$ and $f_2$, positive solutions are guaranteed for $\lambda$ small.
\begin{thm}
\label{thm:local}
Suppose that \eqref{H_inf} holds. Then, there exists $\widetilde{\lambda}>0$ such that for all $\lambda\in(0,\widetilde{\lambda}]$, \eqref{pde} has a positive weak solution $\big(\lambda,(u_1,u_2)\big)$ 
such that 
\begin{equation*}
\|(u_{1}, u_{2})\|_{C(\overline{\Omega})^2}\rightarrow \infty \qq{as} \lambda\rightarrow 0^+.    
\end{equation*} 
Moreover, there exists a connected component of positive weak solutions of \eqref{pde}, namely $\mathscr{C}^+\subset \Sigma$, bifurcating from infinity at $\lambda=0$, such that 
the projection of $\mathscr{C}^+$ on the parameter space is $(0,\widetilde{\lambda}].$ 
\end{thm}

For results concerning the corresponding scalar elliptic problems with nonlinear boundary conditions, see for instance \cite{BCDMP_nonlbc, Ko_Lee_Sim} and references therein. 

Proof of Theorem~\ref{thm:local} is motivated by the re-scaling argument of \cite{ANZ, Amb-Arc-Buf_1994} for the single equation case and extension of their approach to a coupled system of equations in \cite{Chh-Girg-2009}, with Dirichlet boundary conditions. 
This rescaling method  with a multiplicative parameter, 
transforms the original system into a limiting problem with pure-power nonlinearities as the multiplicative parameter approaches zero. We incorporate this re-scaling method combined with Leray-Schauder degree theory to prove the result.
\\
 
\par For our second result, in addition to \eqref{H_inf}, we impose the following conditions on $f_i$ with $i=1,2$ at zero to further guarantee bifurcation from the line of trivial solutions: 
\begin{equation}
\label{H0}
\tag{${\rm H}_0$}
	\begin{cases}
	f_i \in C^1 \\
    f_i(0)=0, f_i^{\prime}(0)>0,\\
    \mbox{ and there exists a constant } \nu_i>1 \mbox{ such that}\\
    \ f_i(s)=f_i^{\prime}(0)s+\mathcal{R}_i(s) \mbox{ for }s\geq 0 \mbox{ with }\mathcal{R}_i(s)=O(s^{\nu_i})\mbox{ as }s\rightarrow 0.
	\end{cases}
  \end{equation}  
We obtain the following global bifurcation result. 
\begin{thm}
\label{thm-2}
Let $f_i$ satisfy \eqref{H0} and \eqref{H_inf}, and there exists $K>0$ such that 
\begin{equation}\label{K}
f_i(s)\geq K s, \quad  \text{ for all }\quad s\geq 0, \quad i=1,2.    
\end{equation}
Then there exists a connected component $\mathscr{C}^+ \subset \Sigma$ of positive weak solution to \eqref{pde}, emanating from the trivial solution at the bifurcation point $\big(\mu_0,(0, 0)\big) \in \Sigma$, where
\begin{equation}\label{def:mu0}
\mu_0:=\dfrac{\mu_1}{\sqrt{f_1^{\prime}(0)f_2^{\prime}(0)}},  
\end{equation}
and possessing a unique bifurcation point from infinity at $\lambda_{\infty}=0$. More precisely, if $\big(\lambda,(u_1, u_2)\big)\in \mathscr{C}^+$, then the following holds:
\begin{equation}\label{unbdd:zero}
\begin{cases}
\|(u_1, u_2)\| \to 0 \ \ &\hspace{-8cm} \mbox{ as } \lambda\to \mu_0,\ \\
\|(u_1, u_2)\| \to \infty &\hspace{-8cm} \mbox{ as } \lambda\to  0^+\,, \text{and} \\
\text{if }\big(\lambda_\infty,(\infty, \infty)\big) \text{ is a bifurcation point from infinity, then } \lambda_\infty= 0.
\end{cases}
\end{equation} 
Moreover, for $\lambda > \frac{\mu_1}{K}$, \eqref{pde} has no positive solution. 
\end{thm}

\par To prove Theorem~\ref{thm-2}, we establish a version of Crandall-Rabinowitz's local bifurcation result \cite{C-R} to obtain {\it bifurcation from zero}, and of Rabinowitz's global bifurcation result \cite{R71} combined with the uniform \emph{a priori} result for the re-scaled solution pair to establish {\it bifurcation from infinity} . We also note that the hypothesis \eqref{K} guarantees nonexistence of positive solution for $\lambda$ large which turns out to be useful in analyzing the global behavior of the branch of positive solutions.


{\it Bifurcation from infinity}  result by Rabinowitz \cite{R73}
applies to problems involving nonlinearities that are asymptotically linear at infinity, that is, when the nonlinear part is a small perturbation of the linear part. This result do not apply to our problem, and hence  we obtain bifurcation from infinity by following the approach via re-scaling technique in \cite{Amb-Arc-Buf_1994, ANZ}. 
\par We refer the reader to, for instance \cite{Castro_Pardo_PJM_2012} for bifurcation from zero results, and \cite{Arrieta-Pardo-RBernal_2007}  for bifurcation from infinity results. 

\medskip
Finally, concerning the multiplicity of solutions, we define the following quantities which are essential for determining the direction of bifurcation  of weak solutions near the bifurcation point. Let $\nu:= \min\{\nu_1,\nu_2\}>1$ and define
\begin{equation}\label{eq:def-Ris}
 \underline{\mathcal{R}}_i:=\liminf_{s\rightarrow 0^+} \dfrac{\mathcal{R}_i(s)}{s^{\nu}}\quad \text{ and }\quad \overline{\mathcal{R}}_i:=\limsup_{s\rightarrow 0^+} \dfrac{\mathcal{R}_i(s)}{s^{\nu}}, \quad i=1, 2.
\end{equation}
Using the hypothesis \eqref{H0}, one can verify that $-\infty < \underline{\mathcal{R}}_i\le \overline{\mathcal{R}}_i<+\infty$.
Now,  define
\begin{equation}
\label{eq:R-zero}
\begin{array}{rlcc}
    \overline{\mathcal{R}}_0&=\frac12 \left(\dfrac{\zeta}{1+\zeta}\right)^{\nu-1}\,\overline{\mathcal{R}}_1+  \frac12  
    \left(\dfrac{1}{1+\zeta}\right)^{\nu-1}
    \,\overline{\mathcal{R}}_2\,,
\end{array}
\end{equation}
where  $\z:=\sqrt{\frac{f_2^{\prime}(0)}{f_1^{\prime}(0)}}$. Similarly, we can define $\underline{\mathcal{R}}_0$ by substituting $\overline{\mathcal{R}}_i$ by $\underline{\mathcal{R}}_i$ for $i=1, 2$.

\begin{thm}\label{thm-3} Suppose that assumptions of Theorem \ref{thm-2} hold, and that $f_i$'s are monotonically non-decreasing. If $\overline{\mathcal{R}}_0 < 0$, then \eqref{pde} has at least two positive weak solutions  in $\left(\mu_0,\overline{\lambda}\right),$ where 
\begin{equation}\label{la:up}
\overline{\lambda}:=\sup\Big\{\lambda\colon \big(\lambda,(u_1,u_2)\big)\in \mathscr{C}^+\Big\} \le \frac{\mu_1}{K}\,.    
\end{equation}
Moreover, \eqref{pde} has at least one weak positive solution for $\lambda=\overline{\lambda}$ and also for $\lambda=\mu_0$.
\end{thm}

To prove this multiplicity result, we use degree theory and the sub- and supersolution method. We apply the  sub- and supersolution method for the cooperative system established in \cite{BLM_maxmin}, hence we require the monotonicity assumption in Theorem \ref{thm-3}.

\medskip

 The scalar versions of Theorem~\ref{thm:local}-Theorem~\ref{thm-3} are considered in \cite{BCDMP_nonlbc}. The extension to the strongly coupled systems considered in this work presents technical challenges. The bootstrap argument for the regularity result in Theorem~\ref{th:bootstrap} demonstrates the challenges present due to the coupling since the analysis requires dealing with different Lebesgue spaces.  Moreover, the bifurcation analysis for Theorem~\ref{thm-2}, requires a Jordan matrix formulation to characterize the linearization satisfying the Crandall--Rabinowitz transversality condition.
 See also \cite{Delgado_Pardo, Fleckinger_Pardo, Fleckinger_Pardo_deThelin}, where a similar linearization approach was utilized for coupled system of equations.

\medskip

The remainder of this paper is organized as follows. In Section~\ref{sec:prelim}, we establish preliminary results including regularity and positivity of weak solutions, and include a uniform \emph{a priori} bounds result of \cite{Bon-Ros_2001}. In particular, we prove that weak solutions are classical solutions through a bootstrap argument in Theorem \ref{th:bootstrap}. Section~\ref{sec:local} is devoted to the proof of Theorem \ref{thm:local} concerning local bifurcation from infinity. We use novel combinations of rescaling methods, tailored to the systems setting, along with the careful degree theory arguments, that take into account the additional complexity introduced by coupling effects, see Theorem \ref{thm:local}. In Section~\ref{sec:global}, we prove the global bifurcation result stated in Theorem \ref{thm-2}, establishing the existence of a connected component of positive solutions bifurcating from the trivial solution and possessing a unique bifurcation point from infinity. We begin by transforming the system into matrix form and applying the Crandall-Rabinowitz bifurcation theorem. Finally, Section~\ref{sec:multiplicity} addresses the multiplicity result of Theorem \ref{thm-3}, where we first determine the bifurcation direction by analyzing the signs of $\underline{\mathcal{R}}_0$ and $\overline{\mathcal{R}}_0$, see \eqref{eq:def-Ris}, and then employ degree theory combined with sub- and supersolution methods to establish the existence of multiple positive solutions when the bifurcation occurs to the right.

\section{Preliminaries and Auxiliary Results}
\label{sec:prelim}

In this section, we discuss the regularity and positivity of weak solutions of \eqref{pde} and state a uniform \emph{a priori} bounds result. Our main result in this section is to prove that weak solutions are, in fact, classical solutions; see Theorem \ref{th:bootstrap}, which may be of independent interest. The proof is achieved through a bootstrap argument.

\subsection{Regularity of weak solutions and positivity}
\label{regularity:linear}
First, we discuss the trace operator to deal with the boundary terms.
The \textit{trace} operator 
\begin{equation}\label{def:trace}
\Gamma\colon W^{1,m}(\Omega)\rightarrow L^r(\partial\Omega), \quad \Gamma(u):=u|_{\partial\Omega}\,,
\end{equation}
ensures that for every $u\in W^{1,m}(\Omega)$ the trace $\Gamma u$ is well defined, that is, 
\begin{equation}\label{cont:Ga}
\begin{cases}
\Gamma u\in L^r(\partial\Omega), \text{ for }r\le \frac{(N-1)m}{N-m}   &\text{if } m<N,\\
\Gamma u\in L^r(\partial\Omega), \text{ for any }r\ge 1 &\text{if } m=N,\\
u\in C^{\alpha}(\overline{\Omega}),  &\text{if } m>N,\quad \alpha=1-\frac{N}{m}\in (0,1)\,.
\end{cases}
\end{equation}
Moreover, by \cite[Ch.~ 6]{KufnerJohnFucik77}, we see that 
\begin{equation}\label{eq:trace-compact}
 \Gamma \text{ is continuous if  }\ \dfrac{N-1}{r}\geq \dfrac{N}{m}-1,   \q{and compact if}\ \dfrac{N-1}{r}> \dfrac{N}{m}-1\,. 
\end{equation}
Next, we discuss the regularity of the solution to a linear problem of the form
\begin{equation}\label{linear}
\left\{
\begin{array}{rcll}
-\Delta v +v&=&0  \quad &\mbox{in}\quad \Omega\,,\\
\frac{\partial v}{\partial \eta} &=&  h\quad &\mbox{on}\quad \partial \Omega,
\end{array}
\right. 
\end{equation}
where $h\in L^q(\partial\Omega)$ for $q\geq 1$. It is well-known that \eqref{linear} has a unique weak solution (see \cite{Ama_1971, Mavinga-Pardo_PRSE_2017}). Let $T\colon L^q(\partial\Omega)\rightarrow W^{1,m}(\Omega),$  $1\leq m\leq Nq/(N-1)$, be the continuous solution operator corresponding to \eqref{linear}, that is, 
\begin{align}
&v=T(h) \quad \text{solves}\quad \eqref{linear}, \text{and}  \nonumber\\
\label{T-bounded}&\text{there exists $C>0$ such that}\\
&\|v\|_{W^{1,m}(\Omega)}\leq C \|h\|_{L^q(\partial\Omega)} \quad \mbox{where} \quad 1\leq m\leq Nq/(N-1).\qquad \nonumber
\end{align}

Now, for any $q\geq 1,$ we 
define the \textit{Neumann-to-Dirichlet operator}  
\begin{equation}\label{D:to:N}
   S: L^q(\p\Om) \to L^r(\p\Om) \quad \text{ given by }\quad  Sh = \Gamma(Th)\,.
\end{equation}
Then, by the inherited properties of the trace operator $\Gamma$, see \eqref{eq:trace-compact},  
\begin{equation*}
S \text{ is continuous if  }\ \dfrac{1}{r}\geq \dfrac{1}{q}-\frac1{N-1},   \q{and compact if}\ \dfrac{1}{r}>\dfrac{1}{q}-\frac1{N-1}\,. 
\end{equation*}

\begin{prop} \label{max_2}
Let $v \in C^2(\Omega) \cap C^1(\overline{\Omega})$ be a solution to \eqref{linear} for $h \geq 0$ with $h \neq 0$. Then $v > 0$ on $\overline{\Omega}$.
\end{prop}
\begin{proof}
Clearly $v>0$ in $\Omega$ by the strong  maximum principle, see \cite[p.~127]{Ama_1971}. We claim that $v>0$ on $\p \Om$ as well.
If not,  there exists  a point $x_0\in\partial\Omega$  such that $v(x_0)=0$. By Hopf's Lemma 
(\cite[Lem.~3.4]{Gil-Trud_2001})
$
\frac{\partial v}{\partial \eta}(x_0)<0,
$ 
contradicting the boundary condition $\frac{\partial v}{\partial \eta}(x_0)=h(x_0) \geq 0$. As a conclusion, $v>0$ on $\overline{\Omega}$.
\end{proof}

\subsection{Weak solution to classical solution}
We say $\big(\lambda,(u_1,u_2)\big)\in (0,\infty)\times (H^1(\Omega))^2$ is a weak solution to the problem \eqref{pde} if
\begin{equation}\label{eq:weak-sol}
\begin{aligned}
\int_{\Omega} \nabla u_1\nabla \psi_1 +\int_{\Omega} u_1\psi_1&=\lambda \int_{\partial\Omega} f_1(u_2)\psi_1\,,\\ 
\int_{\Omega} \nabla u_2\nabla \psi_2 +\int_{\Omega} u_2\psi_2&=\lambda \int_{\partial\Omega} f_2(u_1)\psi_2\,,	
\end{aligned}
\end{equation}
for all $\psi_1, \psi_2 \in H^1(\Omega)$. By \eqref{H_inf},  we get 
\begin{equation*}
\frac{2(N-1)p_1}{N}\le 2_* \mbox{ and }\frac{2(N-1)p_2}{N}\le 2_*\,,    
\end{equation*}
where one of the inequalities is strict. Therefore,  
 $f_1\big(u_2(\cdot)\big),\, f_2\big(u_1(\cdot)\big)\in L^{\frac{2(N-1)}{N}}(\p \Om)$, and hence the  right hand sides of \eqref{eq:weak-sol} are well defined. Clearly, the left hand sides of \eqref{eq:weak-sol} are well defined since $u_i, \psi_i \in H^1(\Omega)$ for $i=1,2$. 
 \par In what follows, we show that weak solutions of \eqref{pde} are, in fact, classical solutions. The result is independent of the parameter $\lambda$, therefore, without loss of generality we can assume $\lambda =1$.

\begin{thm}\label{th:bootstrap}
Let $N> 2$ and $f_i\colon[0,\infty)\rightarrow [0,\infty)$ be  Hölder continuous  satisfying \eqref{H_inf}, for  $i=1,2$. Let $(u_1, u_2)$ be a nontrivial weak solution to

\begin{equation}\label{pde2}
	\left.
	\begin{array}{rclll}
-\Delta u_1+u_1 & = 0 \quad \mbox{in}\quad \Omega\,,&
&\frac{\partial u_1}{\partial \eta}&= f_1(u_2)\quad \mbox {on}\quad \partial\Omega\,;\\
-\Delta u_2 +u_2 &=  0 \quad \mbox{in}\quad \Omega\,, &
&\frac{\partial u_2}{\partial \eta}&= f_2(u_1)\quad \mbox {on}\quad \partial\Omega\,.
\end{array}
	\right\} 
\end{equation}
Then  $u_1, u_2 \in C^{\alpha}(\overline{\Omega})$ for some $\alpha\in(0,1)$, and
 \begin{align}\label{eq:estimate}
\|(u_1,u_2)\|_{(C^{\alpha}(\overline{\Om}))^2}\leq C\Big(1+\big\|\Gamma u_1\|_{L^{2_*}(\p\Om)}^{p_1}+\|\Gamma u_2\big\|_{L^{2_*}(\p\Om)}^{p_2}\Big).
\end{align}
Moreover, $u_1,u_2\in C^{2,\alpha}(\Omega)\cap C^{1,\alpha}(\overline{\Omega})$.
\end{thm}
\begin{proof}
We will first prove that $u_1, u_2 \in C^{\alpha}(\overline{\Omega})$ for some $\alpha\in(0,1)$. Since $u_1, u_2\in H^1(\Omega)$, we have
\begin{equation}\label{eq:bound-trace:0}
u_1\in L^{r_0}(\partial\Omega), \  u_2\in L^{r_0'}(\partial\Omega) \qq{for} r_0=r_0'= 2_*.    
\end{equation}
Due to the hypothesis \eqref{H_inf}, there exists a positive constant $C$ such that 
\begin{equation}\label{eq:bound-trace}
h_1:=  f_1(u_2)\leq C(1+|u_2|^{p_2}), \quad h_2:=f_2(u_1)\leq C(1+|u_1|^{p_1})\,.  
\end{equation}
Then, if $u_1\in L^{r_{i-1}}(\partial\Omega)$ and $u_2\in L^{r_{i-1}^{\prime}}(\partial\Omega),\ i=1,2,\dots,$
\begin{equation}\label{eq:bound-trace:2}
h_1\in L^{q_{i-1}}(\partial\Omega) \q{and} h_2\in L^{q_{i-1}^{\prime}}(\partial\Omega), \q{where} q_{i-1}:=\frac{r_{i-1}}{p_2} \q{and} q_{i-1}^{\prime}:=\frac{r_{i-1}'}{p_1} 
,    
\end{equation}
respectively. Since  $1 < p_1, p_2 \leq \frac{N}{N-2}$ and $N>2$, one has $q_0,\, q_0^{\prime}\ge\frac{2(N-1)}{N}>1$. Now for $h_1\in L^{q_{i-1}}(\partial\Omega)$ and $h_2\in L^{q_{i-1}^{\prime}}(\partial\Omega),\ i=1,2,\dots,$ and using \eqref{T-bounded}, we have
\begin{equation}
\label{eq:bound-trace:3}
  u_1\in W^{1, m_i}(\Omega), \quad  u_2\in W^{1, m_i^{\prime}}(\Omega),\quad \text{where } m_i:=\frac{Nq_{i-1}}{N-1},
   \text{ and } m_i^{\prime}:=\frac{Nq_{i-1}^{\prime}}{N-1}\,.
\end{equation}
Moreover, since $m_i < N$ if and only if $q_{i-1}< N-1$, by \eqref{cont:Ga}, we obtain
\begin{equation}\label{cont:Ga:2}
\begin{cases}
u_1\in L^{r_i}(\partial\Omega), \text{ for } r_i:=\frac{(N-1)q_{i-1}}{N-1-q_{i-1}}   &\text{if } q_{i-1}<N-1,\\
u_1\in L^{r}(\partial\Omega), \text{ for any }r\ge 1 &\text{if } q_{i-1}=N-1,\\
u_1\in C^{\alpha}(\overline{\Omega}),  &\text{if } q_{i-1}>N-1\,,
\end{cases}
\end{equation}
and
\begin{equation}\label{cont:Ga:3}
\begin{cases}
u_2\in L^{r_i'}(\partial\Omega), \text{ for } r_i':=\frac{(N-1)q_{i-1}'}{N-1-q_{i-1}'}   &\text{if } q_{i-1}'<N-1,\\
u_2\in L^{r}(\partial\Omega), \text{ for any }r\ge 1 &\text{if } q_{i-1}'=N-1,\\
u_2\in C^{\alpha}(\overline{\Omega}),  &\text{if } q_{i-1}'>N-1.
\end{cases}
\end{equation}
By \eqref{eq:bound-trace} and the continuity of the Nemytski operator, it follows that for $u_1\in L^{r_i}(\partial \Omega), \ u_2\in L^{r_i^{\prime}}(\partial \Omega)$, 
\begin{equation}\label{cont:Ga:4}
h_1\in L^{q_i}(\partial\Omega), \qquad h_2\in L^{q_i^{\prime}}(\partial\Omega),\qq{where} q_i:=\frac{r_i^{\prime}}{p_2} \text{ and } q_i^{\prime}:=\frac{r_i}{p_1}.
\end{equation}
Summarizing the relations on the exponents of the Lebesgue spaces given by
\eqref{eq:bound-trace:0} - \eqref{cont:Ga:4}, it 
follows that, whenever $q_{i-1},\ q_{i-1}'<N-1$, we get
\begin{align}\label{def:qi:ri}
r_0=r_0'=2_*,\quad
\frac1{q_i}=\frac{p_2}{r_i^{\prime}},\quad  \frac1{q_i^{\prime}}=\frac{p_1}{r_i},
\quad \frac1{r_i}=\frac1{q_{i-1}}-\frac1{N-1},\text{ and } \frac1{r_i'}=\frac1{q_{i-1}^{\prime}}-\frac1{N-1}.
\end{align}
Now, we analyze different cases$\colon$
\begin{enumerate}[a)]
\item If $h_1\in L^{N-1}(\partial\Omega)$, then by \eqref{cont:Ga:2}, $u_1\in L^{r}(\partial\Omega)$ for any $r\ge 1.$
On the other hand, if $h_1\in L^{q_{i-1}}(\partial\Omega)$ with $q_{i-1}> N-1$, then by \eqref{cont:Ga:2} $u_1\in C^{\alpha}(\overline{\Omega})$. 
In both cases, $u_1\in L^{r_{i}}(\partial\Omega),$ for any $r_i \ge 1$. Therefore, $h_2\in L^{q'_{i-1}}(\partial\Omega)$ for any $ q'_{i-1}\ge 1$, in particular for $q'_{i-1}> N-1$. 
Thus, by \eqref{cont:Ga:3}, $u_2\in C^{\alpha}(\overline{\Omega})$. Consequently, $h_1\in L^{q_{i-1}}(\partial\Omega)$ with $q_{i-1} > N-1$, hence $u_1\in C^{\alpha}(\overline{\Omega})$ as well.
Likewise, if $q_{i-1}' \geq N-1$, then $u_1, u_2\in C^{\alpha}(\overline{\Omega})$.
Hence, if either $q_{i-1} \geq N-1$ or $q_{i-1}' \geq N-1,$ for some $i\in \mathbb{N}$, then $u_1, u_2\in C^{\alpha}(\overline{\Omega})$ for some $\alpha\in (0,1)$.
    \item  Suppose $q_i<N-1$ and $q_i^{\prime}<N-1$ for all $i\in \mathbb{N}$. We consider two cases again.    
    \begin{enumerate}
        \item[($b_1$)] Let    
    $p_1< \frac{N}{N-2}$ and $p_2 < \frac{N}{N-2}$.
    Using $p_2 < \frac{N}{N-2}$, \eqref{def:qi:ri}, one gets
\begin{align*}
    \dfrac{1}{r_1}=\dfrac{1}{q_0}-\dfrac{1}{N-1}=\dfrac{p_2}{r_0'}-\dfrac{1}{N-1}<\dfrac{N}{2(N-1)}-\dfrac{1}{N-1}=\dfrac{N-2}{2(N-1)}=\dfrac{1}{r_0}\,.
    \end{align*}

    Analogously, $p_1<\dfrac{N}{N-2}$ implies 
    $
    \dfrac{1}{r_1^{\prime}}<\dfrac{1}{r_0^{\prime}}\,.
    $
    Assume that $r_i > r_{i-1}$ and $r_i' > r_{i-1}'$ for some $i \geq 1$. Then \eqref{def:qi:ri} gives
  \begin{align*}
    \dfrac{1}{r_{i+1}}=\dfrac{1}{q_i}-\dfrac{1}{N-1}=\dfrac{p_2}{r_i'}-\dfrac{1}{N-1}
    < \dfrac{p_2}{r_{i-1}'}-\dfrac{1}{N-1}=\dfrac{1}{q_{i-1}}-\dfrac{1}{N-1}=\dfrac{1}{r_i} \,.
    \end{align*}
By induction, $\{r_i\}$ and $\{r_i^{\prime}\}$ are strictly increasing, and hence  $\{q_i\}$, $\{q_i'\}$, $\{m_i\}$ and $\{m_i^{\prime}\}$ are strictly increasing, by \eqref{cont:Ga:4} and \eqref{eq:bound-trace:3}. 

Therefore, since $q_i$ and $q_i'$ are bounded above by $N-1$, $q_i\rightarrow q_{\infty}$ and $q_i'\rightarrow q_{\infty}'$ for some $1\leq q_{\infty},q_{\infty}'\leq N-1$. If $q_{\infty}=N-1$, 

then  $r_i\to\infty$ by \eqref{cont:Ga:2}, which contradicts \eqref{cont:Ga:4}. Similar argument holds when $q_{\infty}'=N-1$ using \eqref{cont:Ga:3} and \eqref{cont:Ga:4}.  Therefore, $q_{\infty}<N-1$ and $q_{\infty}'<N-1$. Let us define 
     \begin{align}\label{eq:limits}
     \nonumber r_{\infty}:&=\lim_{i\rightarrow \infty}{r_i}=\lim_{i\rightarrow \infty}{p_1\,q_i'}=p_1 q_{\infty}'>0,\\
     r_{\infty}'&:=\lim_{i\rightarrow \infty}{r_i^{\prime}}=\lim_{i\rightarrow \infty}{p_2\,q_i}=p_2 q_{\infty}>0. 
     \end{align}
 By using the first equality of \eqref{cont:Ga:4}  and the third equality of \eqref{def:qi:ri}, we get 
     \begin{align}\label{dif:q}
        q_{i+1}-q_i&=\dfrac{r_{i+1}^{\prime}-r_i^{\prime}}{p_2}=\dfrac{r_{i+1}^{\prime}r_i^{\prime}}{p_2}\left(\dfrac{1}{r_i^{\prime}}-\dfrac{1}{r_{i+1}^{\prime}}\right)\\
        &=\dfrac{r_{i+1}^{\prime}r_i^{\prime}}{p_2}\left(\dfrac{1}{q_{i-1}'}-\dfrac{1}{q_i'}\right)=\dfrac{r_{i+1}^{\prime}r_i^{\prime}}{p_2}\left(\dfrac{q_i'-q_{i-1}'}{q_{i-1}'q_i'}\right).\nonumber
     \end{align}
     Then, dividing \eqref{dif:q} by $q_i'-q_{i-1}'$ and using \eqref{cont:Ga:4} again, we get
     \begin{equation}\label{eq:quotient-1}
        \dfrac{q_{i+1}-q_i}{q_i'-q_{i-1}'}=\dfrac{r_{i+1}^{\prime}r_i^{\prime}}{p_2q_{i-1}'q_i'}=\dfrac{p_1^2}{p_2}\, \dfrac{r_{i+1}^{\prime}\, r_i'}{r_{i-1}\, r_i}. 
     \end{equation}
     Analogously,
      \begin{equation}\label{eq:quotient-2}
                 \dfrac{q_{i+1}'-q_i'}{q_i-q_{i-1}}=\dfrac{p_2^2}{p_1}\, \dfrac{r_{i+1}\, r_i}{r_{i-1}'\, r_i'}\,.
      \end{equation} 
   Multiplying \eqref{eq:quotient-1} and \eqref{eq:quotient-2} and then taking the limit and using \eqref{eq:limits}, we get
     \[
     \lim_{i\rightarrow \infty}\left(\dfrac{q_{i+1}-q_i}{q_i-q_{i-1}}\right)\left(\dfrac{q_{i+1}'-q_i'}{q_i'-q_{i-1}'}\right)=p_1p_2>1\,,
     \]
a contradiction to the fact that $\{q_i\}$ and $\{q_i'\}$ are convergent sequences and so Cauchy sequences.

\medskip
     
Thus, there exists $i_0\in \mathbb{N}$ such that either $q_{i_0}\geq N-1$ or $q'_{i_0}\geq N-1$ and therefore, $u_1,u_2\in C^{\alpha}(\overline{\Omega})$ for some $\alpha\in (0,1)$.
\item[($b_2$)] Without loss of generality, assume that $p_1 < \frac{N}{N-2} $ and $\ p_2 = \frac{N}{N-2}$. Then, it can be shown that
$$
r_0 = r_1 < r_2 = r_3 < r_4 \cdots \mbox{ and }
r_0' < r_1' = r_2' < r_3' = r_4' \cdots\,.
$$
We claim that if
$
r_{2n}=r_{2n+1} < r_{2n+2}
\quad \mbox{and} \quad
r'_{2n} < r'_{2n+1} = r'_{2n+2}\,,
$
then 
$$
i)\ r_{2n+2}=r_{2n+3},\ \ ii)\  r'_{2n+2}<r'_{2n+3},\ \ iii)\ r_{2n+3}<r_{2n+4},\ \  iv)\ r'_{2n+3}=r'_{2n+4}\,.
$$
Indeed, using \eqref{cont:Ga:4} and \eqref{def:qi:ri}, we have
\begin{align*}
    \dfrac{1}{r_{2n+3}}=\dfrac{1}{q_{2n+2}}-\dfrac{1}{N-1}=\dfrac{p_2}{r_{2n+2}'}-\dfrac{1}{N-1}
    = \dfrac{p_2}{r_{2n+1}'}-\dfrac{1}{N-1}=\dfrac{1}{q_{2n+1}}-\dfrac{1}{N-1}=\dfrac{1}{r_{2n+2}} \,,    
    \end{align*}
\begin{align*}
    \dfrac{1}{r'_{2n+3}}=\dfrac{1}{q'_{2n+2}}-\dfrac{1}{N-1}=\dfrac{p_1}{r_{2n+2}}-\dfrac{1}{N-1}
    < \dfrac{p_1}{r_{2n+1}}-\dfrac{1}{N-1}=\dfrac{1}{q'_{2n+1}}-\dfrac{1}{N-1}=\dfrac{1}{r'_{2n+2}} \,.    
\end{align*}
Similarly, iii) and iv) follow. 

By induction, $\{r_{2i}\}$, $\{r_{2i+1}\}$, $\{r_{2i}^{\prime}\}$ and $\{r_{2i+1}^{\prime}\}$ are strictly increasing, and consequently $\{q_{2i}\}$, $\{q_{2i+1}\}$, $\{q_{2i}^{\prime}\}$ and $\{q_{2i+1}^{\prime}\}$ are strictly increasing by \eqref{cont:Ga:4}. Reasoning as in \eqref{dif:q}, we get
     \begin{align*}
        q_{i+2}-q_i&=\dfrac{r_{i+2}^{\prime}-r_i^{\prime}}{p_2}=\dfrac{r_{i+2}^{\prime}r_i^{\prime}}{p_2}\left(\dfrac{1}{r_i^{\prime}}-\dfrac{1}{r_{i+2}^{\prime}}\right)\\
        &=\dfrac{r_{i+2}^{\prime}r_i^{\prime}}{p_2}\left(\dfrac{1}{q_{i-1}'}-\dfrac{1}{q_{i+1}'}\right)=\dfrac{r_{i+2}^{\prime}r_i^{\prime}}{p_2}\left(\dfrac{q_{i+1}'-q_{i-1}'}{q_{i-1}'q_{i+1}'}\right).\nonumber
     \end{align*}
Proceeding as in the case $(b_1)$, we arrive at the contradiction 
 \[
     \lim_{i\rightarrow \infty}\left(\dfrac{q_{i+2}-q_i}{q_{i+1}-q_{i-1}}\right)\left(\dfrac{q_{i+1}'-q_i'}{q_{i+1}'-q_{i-1}'}\right)=p_1p_2>1\,.
     \]

Thus, as in part $(b_1)$, there exists $i_0\in \mathbb{N}$ such that either $q_{i_0}\geq N-1$ or $q'_{i_0}\geq N-1$ and therefore, $u_1,u_2\in C^{\alpha}(\overline{\Omega})$ for some $\alpha\in (0,1)$.     
\end{enumerate}
\end{enumerate}
 Therefore, we conclude that $u_i=T h_i\in C^{\alpha}(\overline{\Om})$, $i=1,2$ with $\|u_1\|_{C^{\alpha}(\overline{\Om})}\leq C\|h_1\|_{L^{q_0}(\p\Om)}$, $\|u_2\|_{C^{\alpha}(\overline{\Om})}\leq C\|h_2\|_{L^{q_0'}(\p\Om)}$ for some $\alpha\in (0,1)$. Using the facts that  $p_2q_0=r_0=2_*=r'_0=p_1q_0'$ (see \eqref{eq:bound-trace:0} and \eqref{eq:bound-trace:2} for $i=1$), and \eqref{eq:bound-trace}, we get
\begin{align*} 
\|u_1\|_{C^{\alpha}(\overline{\Om})}\leq C\|f_1(u_2)\|_{L^{q_0}(\p\Om)}\leq C\|1+|\Gamma u_2|^{p_2}\|_{L^{q_0}(\p\Om)}\leq C(1+\|\Gamma u_2\|_{L^{2_*}(\p\Om)}^{p_2}),\\
\|u_2\|_{C^{\alpha}(\overline{\Om})}\leq C\|f_2(u_1)\|_{L^{q_0'}(\p\Om)}\leq C\|1+|\Gamma u_1|^{p_1}\|_{L^{q_0'}(\p\Om)}\leq C(1+\|\Gamma u_1\|_{L^{2_*}(\p\Om)}^{p_1}),
\end{align*}
which implies \eqref{eq:estimate}. 
\medskip

Then, using the facts that  $u_i \in C^{\alpha}(\overline{\Omega})$ and $f_i$ are locally Hölder  continuous for $i=1,2$, 
it follows that  $f_1(u_2)$ and $f_2(u_1)$ are bounded. Therefore, 
it follows from \cite[Theorem 2]{Lieberman_1988}
that 
 $u_1,u_2\in C^{1,\al}(\overline{\Omega})$ for some $\al=\al(\g,N) \in (0, 1)$. Finally, by interior elliptic regularity to equations in $\Omega$ of \eqref{pde2}, it follows that $u_1,u_2\in C^{2,\al}(\Omega)\cap C^{1,\al}(\overline{\Omega})$. This completes the proof. 
\end{proof}
The following uniform \emph{a priori} bound result from \cite[Thm.\ 3.7]{Bon-Ros_2001} for strongly coupled systems is crucial for applying degree theory as well as to establish that the branch of positive weak solutions indeed bifurcate from infinity at $\la_{\infty}=0$. We note that all the coefficients $a, b, c, d$ in the hypothesis (H3) of \cite{Bon-Ros_2001} do not need to be strictly positive. The proof in \cite[Thm.~3.7]{Bon-Ros_2001} holds even when $a, b, c, d $ are  non-negative functions with $a(x)+ b(x)\geq k>0$ and $c(x)+d(x) \geq k>0$. Therefore, it holds for our case $a, b, c, d \in \mathbb{R}$ with $a, d  >0$, and $b=c=0$, that is, when $f_i$ satisfy the hypothesis \eqref{H_inf}.
\begin{prop}
 \label{th:Rossi}
 Suppose that the nonlinearities $f_i$ satisfy \eqref{H_inf}, $i=1,2$. 
 Then there exists a constant $M>0$ such that every nonnegative solution $(u_1,u_2)$ of \eqref{pde} satisfies 
 \begin{equation*}
     \|(u_1,u_2)\|_{C(\overline{\Omega})^2}\leq M,
 \end{equation*}
 where $M$ is independent of the solution $(u_1, u_2)$.
\end{prop}
\section{Proof of Theorem~\ref{thm:local}}
\label{sec:local}

To prove Theorem \ref{thm:local}, we follow the ideas of the corresponding scalar case \cite{BCDMP_nonlbc} and the case of the coupled system of equations for the case of the zero Dirichlet boundary condition \cite{Chh-Girg-2009, Chh-Gir-2013}. Specifically,  we introduce rescaled variables, involving the multiplicative parameter $\lambda$, which connects to the case of pure powers in the limit. The limiting problem case is dealt with using the Leray-Schauder degree theory, which then enables us to use homotopy to connect with the rescaled and hence the original problem.
\subsection{Rescaling}
For a given positive solution $(u_1,u_2)$ of \eqref{pde}, we consider the following rescaled functions 
\begin{equation}
    \label{eq:re-scaling}
    w_1=\lambda^{\theta_1}u_1 \quad \text{ and } \quad w_2=\lambda^{\theta_2}u_2\,,
\end{equation}
 where $\theta_1,\theta_2>0$ satisfy 
\begin{equation}
\label{theta_eq}
1+\theta_2-\theta_1p_1=0 \mbox{ and }
1+\theta_1-\theta_2p_2=0\,.
\end{equation}
Then, from \eqref{theta_eq}, it follows that for $\lambda>0$,  $(w_1, w_2)$ satisfy 
\begin{equation}
    \label{pde_rscld}
\left.
	\begin{array}{rclll}
-\De w_1+w_1 & = 0 \quad \mbox{in}\quad \Omega\,,&
&\frac{\partial w_1}{\partial \eta}&=\lambda^{1+\theta_1} f_1(\lambda^{-\theta_2}w_2)\quad \mbox {on}\quad \partial\Omega\,;\\
-\Delta w_2 +w_2 &=  0 \quad \mbox{in}\quad \Omega\,, &
&\frac{\partial w_2}{\partial \eta}&=\lambda^{1+\theta_2} f_2(\lambda^{-\theta_1}w_1)\quad \mbox {on}\quad \partial\Omega\,.
\end{array}
	\right\} 
\end{equation}
Observe that for any $\lambda>0$, $(w_1,w_2)$ is a solution to \eqref{pde_rscld} if and only if $(u_1,u_2)$ is a solution to \eqref{pde}. 
Now for any $s >0$, we define
\begin{align}
 \widetilde{f}_1(\lambda,s) & :=\lambda^{1+\theta_1}f_1(\lambda^{-\theta_2}s) = \lambda^{1+\theta_1}\big[f_1(\lambda^{-\theta_2}s)-b_2\big(\lambda^{-\theta_2}s\big)^{p_2}\big]+ 
 b_2s^{p_2} ,\notag
 \end{align} 
 and 
 \begin{align}
 \widetilde{f}_2(\lambda,s) & :=\lambda^{1+\theta_2}f_2(\lambda^{-\theta_1}s) = \lambda^{1+\theta_2}\big[f_2(\lambda^{-\theta_1}s)-b_1\big(\lambda^{-\theta_1}s\big)^{p_1}\big]+ 
 b_1s^{p_1}\,. \notag
 \end{align} 
Then, since $\theta_i > 0$, $\lambda^{-\theta_i}s \to \infty$ as $\lambda \to 0^+$ for any $s>0$  and $i=1, 2$. Hence, using hypothesis \eqref{H_inf} in the  definitions of $\widetilde{f}_1, \widetilde{f}_2$,  
we get $\widetilde{f_1}(0,s)=b_2s^{p_2}$ and
$\widetilde{f}_2(0,s)=b_1s^{p_1}$. Consequently, as $\lambda \to 0^+$, \eqref{pde_rscld} reduces to the following {limiting problem}: 
\begin{equation}
\label{pde_limtng}
\left.
	\begin{array}{rclll}
-\De w_1+w_1 & = 0 \quad \mbox{in}\quad \Omega\,,&
&\frac{\partial w_1}{\partial \eta}&=b_2w_2^{p_2}\quad \mbox {on}\quad \partial\Omega\,;\\
-\Delta w_2 +w_2 &=  0 \quad \mbox{in}\quad \Omega\,, &
&\frac{\partial w_2}{\partial \eta}&=b_1w_1^{p_1}\quad \mbox{on}\quad \partial\Omega\,.
\end{array}
	\right\} 
\end{equation}
Now we establish the existence of a nontrivial solution to the limiting problem \eqref{pde_limtng} using degree theory.

\subsection{A Compact Operator}
\label{subsec:3.2} Here we construct a compact operator, associated with the rescaled problem \eqref{pde_rscld}, to apply the Leray-Schauder degree theory. 
Now for each $\lambda >0$ fixed, let us define the operator $\widetilde{F}\colon [0,\infty)\times (C(\partial\Omega))^2 \rightarrow (C(\partial\Omega))^2$ as follows
\begin{equation*}
    \widetilde{F}(\lambda,(z_1,z_2)):=
    \begin{pmatrix} \widetilde{f_1}(\lambda,
    z_2)\\
    \widetilde{f_2}(\lambda, 
    z_1)\end{pmatrix}.
\end{equation*}
$\widetilde{F}$ is a continuous operator. Next, we define $\widetilde{G}\colon [0,\infty)\times (C(\partial\Omega))^2 \rightarrow (C(\partial\Omega))^2$ by
\begin{equation*}
\widetilde{G}(\lambda,(z_1,z_2)):= S\circ \widetilde{F}(\lambda,(z_1,z_2))=\begin{pmatrix} S\widetilde{f_1}(\lambda,
    z_2)\\
    S\widetilde{f_2}(\lambda,
    z_1)\end{pmatrix}\,,
\end{equation*}
where $S:=\Ga\circ T$ defined in \eqref{D:to:N} is the Neumann-to-Dirichlet operator, $\Ga$ is the trace operator (see \eqref{def:trace}-\eqref{cont:Ga}), and $T$ is the solution operator of the linear problem \eqref{linear} as defined in \eqref{T-bounded}.  
Using the linear theory in Subsection~\ref{regularity:linear}, by 
choosing $q>N-1$, then $m=\frac{Nq}{N-1}>N$, and it follows that 
\begin{align*}
 (C(\p\Om))^2 \stackrel{\widetilde{F}}{\longrightarrow} (C(\p\Om))^2 \stackrel{i}{\longrightarrow} (L^q(\p\Om))^2  \stackrel{T}{\longrightarrow} (W^{1,m}(\Om))^2 \stackrel{\Gamma}{\hookrightarrow}(C(\p\Om))^2,
 \end{align*} 
where $i$ is the natural injection. Hence $\widetilde{G}$ is a compact operator.

Then, by Theorem \ref{th:bootstrap}, and the definition of the solution set given in \eqref{solution:set}, 
solution to \eqref{pde_rscld} is a fixed point of the operator $\widetilde{G}$. Indeed, for $\lambda>0$
\begin{equation}
\label{eq:fixed-point}
\widetilde{G}\big(\lambda,(\Gamma u_1,\Gamma u_2)\big) = (\Gamma u_1, \Gamma u_2) \iff  \big(\lambda, (u_1,u_2)\big) \mbox{ is a weak solution to \eqref{pde}} .
\end{equation}
Moreover, the continuity of $\widetilde{G}$ guarantees a solution to the limiting problem at $\lambda =0$. Indeed,
\begin{equation*}
\widetilde{G}\big(0,(\Gamma u_1,\Gamma u_2)\big) = (\Gamma u_1, \Gamma u_2) \iff  \big(0, (u_1,u_2)\big) \mbox{ is a weak solution to  \eqref{pde_limtng}} \,.
\end{equation*}

\subsection{Solution to Rescaled Problem}

The following lemma shows that the limiting problem \eqref{pde_limtng} has a nontrivial solution.
\begin{lem} \label{deg_new}
There exist $0 < r < R$ such that the limiting problem \eqref{pde_limtng} has no solution whenever $\|(w_1,w_2)\|_{(C(\p\Om))^2}=r$ and $\|(w_1,w_2)\|_{(C(\p\Om))^2}=R$, and 
\begin{equation*}
\deg (I-\widetilde{G}(0,\cdot), B_R (0) \setminus \overline{B_r}(0),0)=-1\,.     
\end{equation*}
\end{lem}

\begin{proof}
The proof follows from \cite[Thm.\ 3.2]{Bon-Ros_2001}, where authors consider a more general system which include the limiting system \eqref{pde_limtng}. One can verify that the power nonlinearities in \eqref{pde_limtng} satisfy the conditions (H1) and (H2.iii) of \cite[Page 6]{Bon-Ros_2001}. Hence, \eqref{pde_limtng} has a nontrivial solution in $B_R (0) \setminus \overline{B_r}$.
\end{proof}


Now we will use Lemma~\ref{deg_new} and $\lambda \geq 0$ as a homotopy parameter to establish the following result, which guarantees the existence of a positive weak solution to the rescaled problem \eqref{pde_rscld}. 
\begin{lem}\label{rscld}
There exists $\widetilde{\lambda}>0$ such that 
\begin{itemize}
    \item[(a)] $\widetilde{G}(\lambda, (w_1,w_2)) \neq (w_1,w_2)$ for all  $\lambda \in [0,\widetilde{\lambda}]$ whenever $\|(w_1,w_2)\|_{(C(\p\Om))^2}\in \{r,R\}$.
    \item[(b)] $\deg (I-\widetilde{G}(\lambda,\cdot),B_R \backslash \overline{B_r},0)=-1$ for all $\lambda \in [0,\widetilde{\lambda}]$. 
\end{itemize}
\end{lem}
\begin{proof}
Part (a): \\
Suppose not. Then, there exists a sequence $\{\lambda_n \} \in [0, +\infty)$ with $\lambda_n \to 0$ and corresponding $\{(w_{1,n},w_{2,n})\} \in (C(\p\Om))^2$  such that 
$\widetilde{G}(\lambda_n,(w_{1,n},w_{2,n}))=(w_{1,n},w_{2,n})$, and $\|(w_{1,n},w_{2,n})\|_{(C({\p\Omega}))^2}$ $=r$ (or $\|(w_{1,n},w_{2,n})\|_{(C({\p\Omega}))^2}=R$). Using the facts that  $\|(w_{1,n},w_{2,n})\|_{(C({\p\Omega}))^2}=r$ (or $R$)
and $\widetilde{G}$ is a compact operator,  $\widetilde{G}(\lambda_n,(w_{1,n},w_{2,n}))= (w_{1,n},w_{2,n}) $ is convergent up to a subsequence. Therefore, $(\lambda_n,(w_{1,n},w_{2,n})) \to (0,(w_1,w_2))$ as $n \to \infty$, where $(w_1,w_2)\in (C({\p\Omega}))^2$, and the continuity of $\widetilde{G}$ implies $\widetilde{G}(0,(w_1,w_2))=(w_1,w_2)$. That is, $(w_1, w_2)$ is a weak solution to the limiting problem \eqref{pde_limtng} with $\|(w_1,w_2)\|_{(C({\p\Omega}))^2}=r$ (or $\|(w_1,w_2)\|_{(C({\p\Omega}))^2}=R$), which contradicts Lemma \ref{deg_new}.
\\ 

\noindent
Part (b):\\
Part (a) implies that $\deg (I-\widetilde{G}(\lambda,\cdot),B_R \backslash \overline{B_r},0)$ is well defined for all $\lambda \in [0,\widetilde{\lambda}]$. Now using $\lambda \in [0,\widetilde{\lambda}]$ as the homotopy parameter, and using Lemma \ref{deg_new}, we get $\deg (I-\widetilde{G}(\lambda,\cdot),B_R \backslash \overline{B_r},0)=\deg (I-\widetilde{G}(0,\cdot),B_R \backslash \overline{B_r},0)=-1\,.$ This completes the proof.
\end{proof}

Now, we complete the proof of Theorem~\ref{thm:local}. 
By Lemma~\ref{rscld}, the rescaled problem \eqref{pde_rscld} has a non trivial solution $(w_{1},w_{2}) \in \big(C(\overline{\Omega})\big)^2$ for all $\lambda \in [0, \widetilde{\lambda}]$ satisfying $r<\|(w_{1},w_{2})\|_{(C(\partial\Omega))^2}<R$. Since $f_i$\,'s are  nonnegative and satisfy $\eqref{H_inf}$, and so do the $\widetilde{f}_i$\,'s, therefore, $w_i>0$ in $\overline{\Omega}$ by Proposition~\ref{max_2} for $i=1,2$. 
Using Lemma~\ref{rscld}, it follows from  \cite[Prop.~2.3]{deF-Lion_Nus_1982} 
that the rescaled problem \eqref{pde_rscld} has a connected component $\mathscr{D}$ of positive weak solutions along which $\lambda$ takes all the values in $[0,\widetilde{\lambda}]$. From the rescaling \eqref{eq:re-scaling}, it follows that for $\lambda >0$, $(\la,(u_{1},u_{2}))$ is a solution to \eqref{pde}. Also, using $\|(w_1,w_2)\|_{(C(\partial\Omega))^2}>r>0$ and  $\theta_1,\theta_2>0$, we see that $\|(u_{1},u_{2})\|_{(C(\partial\Omega))^2} \to \infty$ as $\lambda \to 0^+$. Additionally, since  $\|(u_{1},u_{2})\|_{C(\partial\Omega)} \le \|(u_{1},u_{2})\|_{C(\overline{\Omega})}$, we get 

$
\|(u_{1},u_{2})\|_{(C(\overline{\Omega})^2}\to\infty$ as $\lambda \to 0^+$.
Moreover, there exists a connected component $\mathscr{C}^+$ of positive weak solutions of \eqref{pde} bifurcating from infinity at $\lambda_{\infty}=0$ such that the projection of $\mathscr{C}^+ \subset \Sigma$ on the parameter space is $(0, \widetilde{\lambda}]$. This completes the proof. \qed

\section{Global Bifurcation: Proof of Theorem \ref{thm-2}}
\label{sec:global}
In this section, we will use bifurcation theory to prove our result. For the sake of notational simplicity, first we rewrite the problem \eqref{pde} in matrix form, which takes advantage of the hypothesis \eqref{H0} near zero as follows:
\begin{align}\label{eq:vector-form}
&\begin{pmatrix}
-\Delta+1 & 0\\
0  &-\Delta+1
\end{pmatrix}
\begin{pmatrix} u_1\\u_2\end{pmatrix}=\begin{pmatrix} 0\\0\end{pmatrix} \quad \text{ in }\Omega\,,\\
&\begin{pmatrix} \frac{\partial u_1}{\partial \eta}\\
\frac{\partial u_2}{\partial \eta}\end{pmatrix}=\lambda A \begin{pmatrix} u_1\\u_2\end{pmatrix} +\lambda\begin{pmatrix} \mathcal{R}_1(u_2)\\\mathcal{R}_2(u_1)\end{pmatrix} \quad \text{ on }\p\Omega\,,\label{eq:vector-form-boundary}
\end{align}
where $A$ is the $2\times 2$ matrix
\begin{align}\label{eq:matrix-A}
A=\begin{pmatrix}0&f_1^{\prime}(0)\\
f_2^{\prime}(0)&0\end{pmatrix}.
\end{align}
Clearly, the  eigenvalues of $A$ are $\{\s,-\s\}$, where 
\begin{equation}\label{s+-}
\sigma:=\sqrt{f_1^{\prime}(0)f_2^{\prime}(0)}\,.  \end{equation}
Therefore, there exists an invertible matrix $P$ such that  
\begin{equation}\label{eq:diagonal}
P^{-1}AP=J:=
\begin{pmatrix}
   \sigma&0\\
   0&-\s 
   \end{pmatrix},
\end{equation}
where 
\begin{equation*}
P=\begin{pmatrix}
\dfrac{1}{1+\z}& \dfrac{1}{1+\z}\\[.3cm]
\dfrac{\z}{1+\z}&\dfrac{-\z}{1+\z}
\end{pmatrix},
\quad\quad
P^{-1}=\begin{pmatrix}
\dfrac{1+\z}{2}& \dfrac{1+\z}{2\z}\\[.3cm]
\dfrac{1+\z}{2}& -\dfrac{1+\z}{2\z}
\end{pmatrix}\,,
\end{equation*}
with  $\z:=\sqrt{\frac{f_2^{\prime}(0)}{f_1^{\prime}(0)}}$. 
Note that $J$ is the Jordan canonical form of the matrix $A$.
Now, multiplying \eqref{eq:vector-form} and \eqref{eq:vector-form-boundary} by $P^{-1}$, and denoting  
\begin{equation}
\label{relation:w:u}
\begin{pmatrix} w_1\\w_2\end{pmatrix}:= P^{-1}\begin{pmatrix} u_1\\u_2\end{pmatrix}
\end{equation}
we obtain the  following system that is equivalent to \eqref{eq:vector-form}-\eqref{eq:vector-form-boundary}.
\begin{align}\label{eq:reduction}
&\begin{pmatrix}
-\Delta+1 & 0\\
0 & -\Delta+1
\end{pmatrix}
\begin{pmatrix} w_1\\w_2\end{pmatrix}=\begin{pmatrix} 0\\0\end{pmatrix} \quad \text{ in }\Omega\,,\\
& \begin{pmatrix} \dfrac{\partial w_1}{\partial \eta}\\ 
\dfrac{\partial w_2}{\partial \eta}\end{pmatrix}
=\la J \begin{pmatrix} w_1\\w_2\end{pmatrix}+ \la \begin{pmatrix}\mathcal{R}^*_1(w_1,w_2) \\ \mathcal{R}^*_2(w_1,w_2) \end{pmatrix}\quad \text{ on }\p\Omega\,,\label{eq:reduction-boundary}
\end{align}
where
\begin{equation}\label{R*}
\begin{pmatrix}\mathcal{R}^*_1(w_1,w_2) \\ \mathcal{R}^*_2(w_1,w_2) \end{pmatrix}:= 
P^{-1}\begin{pmatrix} \mathcal{R}_1(u_2)\\
\mathcal{R}_2(u_1)\end{pmatrix}
=\frac{1+\z}{2}\begin{pmatrix} \mathcal{R}_1(u_2)+\frac{1}{\z}\mathcal{R}_2(u_1)\\ \mathcal{R}_1(u_2)-\frac{1}{\z}\mathcal{R}_2(u_1)\end{pmatrix}\,.
\end{equation}
In the proposition below we establish that $(\mu_0, (0,0))$ is a bifurcation point from the trivial solution.
\begin{prop}\label{prop:limit}
Suppose that $f_1$ and $f_2$ satisfy \eqref{H0}. Let $\{\lambda_n\}$ be a convergent sequence of real numbers and $\{(u_{1,n}, u_{2,n})\}$ be the corresponding sequence of positive solutions of \eqref{pde} satisfying $\|(u_{1,n}, u_{2,n})\|\rightarrow 0$ as $n\rightarrow \infty$. Then
\begin{equation}\label{bif:pt}
  \lambda_n\rightarrow \mu_0 \qq{with} \mu_0 \text{ is as defined in } \eqref{def:mu0}\,,
\end{equation}
 and, up to a subsequence,
 \begin{align}\label{phi1}
 \frac{u_{1,n}}{\|(u_{1,n}, u_{2,n})\|} &\rightarrow \phi_1:=\frac{\varphi_1}{1+\sqrt{f_2^{\prime}(0)/f_1^{\prime}(0)}},\\
 \frac{u_{2,n}}{\|(u_{1,n}, u_{2,n})\|}&\rightarrow \phi_2:=\sqrt{f_2^{\prime}(0)/f_1^{\prime}(0)}\ \phi_1=\frac{\varphi_1}{1+\sqrt{f_1^{\prime}(0)/f_2^{\prime}(0)}} \label{phi2},
 \end{align}
 in $C^{\beta}(\overline{\Omega})$, 
where $\mu_1$, $\varphi_1$ are the first Steklov eigenvalue and the corresponding eigenfunction, respectively,  defined in \eqref{steklov}. 
\end{prop}
\begin{proof}
Suppose that $\lambda_n\rightarrow\underline{\lambda}$ for some $\underline{\lambda}\in\mathbb{R}$ and define
\[
     (v_{1,n},v_{2,n}):=\left(\dfrac{u_{1,n}}{\|(u_{1,n}, u_{2,n})\|}\,, \dfrac{u_{2,n}}{\|(u_{1,n}, u_{2,n})\|}\right).
\]
Then $(v_{1,n},v_{2,n})$ is a weak solution to the following problem
\begin{equation}
\label{pde-3}
\left.
\begin{array}{rclll}
-\De v_{1,n}+v_{1,n} & = 0 \ \mbox{ in}\ \Omega\,,&
&\frac{\partial v_{1,n}}{\partial \eta}&= \lambda_n\left(f_1^{\prime}(0)v_{2,n}+\dfrac{\mathcal{R}_1(u_{2,n})}{\|(u_{1,n},u_{2,n})\|}\right)\ \mbox { on}\ \partial\Omega\,;\\
-\Delta v_{2,n} +v_{2,n} &=  0 \ \mbox{ in}\ \Omega\,, &
&\frac{\partial v_{2,n}}{\partial \eta}&=\lambda_n\left(f_2^{\prime}(0)v_{1,n}+\dfrac{\mathcal{R}_2(u_{1,n})}{\|(u_{1,n},u_{2,n})\|}\right) \ \mbox{ on}\ 
\partial\Omega\,,
\end{array}
\right\} 
\end{equation}
where $\mathcal{R}_i$ are the corresponding remainders of $f_i$ for $i=1,2$ given in \eqref{H0}. Since $\|u_{i,n}\|_{C(\overline{\Omega})}  \le \|(u_{1,n},u_{2,n})\|_{(C(\overline{\Omega}))^2}$ for $i=1, 2$, 
we have
\begin{equation}\label{eq:reminders}
\dfrac{|\mathcal{R}_1(u_{2,n})|}{\|(u_{1,n},u_{2,n})\|}\rightarrow 0, \quad \dfrac{|\mathcal{R}_2(u_{1,n})|}{\|(u_{1,n},u_{2,n})\|}\rightarrow 0, \quad \mbox{as}\quad n\rightarrow \infty
\end{equation}
which implies that the right-hand side of \eqref{pde-3} is bounded in $\big(C(\overline{\Omega})\big)^2$. 

On one hand, by \eqref{T-bounded} we obtain that $(v_{1,n},v_{2,n})\in (W^{1,s}(\Omega))^2$ for $s>1$. In particular, $(v_{1,n},v_{2,n})\in (W^{1,m}(\Omega))^2$ for $m>N$. By the Sobolev embedding theorem, $(W^{1,m}(\Omega))^2\hookrightarrow (C ^{\alpha}(\overline{\Omega}))^2$ for $\alpha>1-N/m$. Moreover,  $(C^{\alpha}(\overline{\Omega}))^2\hookrightarrow (C ^{\beta}(\overline{\Omega}))^2$ for $0<\beta<\alpha$ (compactly embedded). Consequently, up to a subsequence, $(v_{1,n},v_{2,n})\rightarrow (\phi_1,\phi_2)$ in $(C ^{\beta}(\overline{\Omega}))^2$ and $\phi_1, \phi_2\geq 0$. Further, since $\|(v_{1,n},v_{2,n})\|_{(C(\overline{\Omega}))^2}=1$, we obtain that $\|(\phi_1,\phi_2)\|_{(C(\overline{\Omega}))^2}=1$.

On the other hand, in particular, $(v_{1,n},v_{2,n})\in (H^1(\Omega))^2$ and $\{(v_{1,n},v_{2,n})\}$ is uniformly bounded. Therefore, due to the fact that the product space $(H^1(\Omega))^2$ is reflexive,  $\{(v_{1,n},v_{2,n})\}$ has a weakly convergent subsequence, namely $(v_{1,n},v_{2,n})\rightharpoonup (v_1,v_2)$ in $(H^1(\Omega))^2$ which in fact converges strongly $(v_{1,n},v_{2,n})\to (v_1,v_2)$ in $(L^2(\Omega))^2$. Observe that, the weak formulation of \eqref{pde-3} can be represented as follows:
\begin{equation}
	\label{weak-limit}
	\begin{array}{rcll}
\int_{\Omega} \nabla v_{1,n}\nabla \psi_1 +\int_{\Omega} v_{1,n}\psi_1&=&\lambda_n\scaleint{7ex}_{\hspace{-.3cm}\p\Om} \left(f_1^{\prime}(0)v_{2,n}+\dfrac{\mathcal{R}_1(u_{2,n})}{\|(u_{1,n},u_{2,n})\|}\right)\psi_1,\\[.5cm]
\int_{\Omega} \nabla v_{2,n}\nabla \psi_2 +\int_{\Omega} v_{2,n}\psi_2&=&\lambda_n\scaleint{7ex}_{\hspace{-.3cm}\p\Om} \left(f_2^{\prime}(0)v_{1,n}+\dfrac{\mathcal{R}_2(u_{1,n})}{\|(u_{1,n},u_{2,n})\|}\right)\psi_2,	\end{array}
\end{equation}
for all $\psi_1, \psi_2 \in H^1(\Omega)$.
Now, using the weak convergence $(v_{1,n},v_{2,n})\rightharpoonup (v_1,v_2)$ in $(H^1(\Omega))^2$ as $n \to \infty$, on the LHS of \eqref{weak-limit}, we get
\begin{equation}
\label{eq:limit-1}
\begin{array}{rcll}
\lim_{n\rightarrow \infty}\int_{\Omega} \nabla v_{1,n}\nabla \psi_1 +\int_{\Omega} v_{1,n}\psi_1&=\int_{\Omega} \nabla v_1\nabla \psi_1 +\int_{\Omega} v_1\psi_1,\\
\lim_{n\rightarrow \infty}\int_{\Omega} \nabla v_{2,n}\nabla \psi_2 +\int_{\Omega} v_{2,n}\psi_2&=\int_{\Omega} \nabla v_2\nabla \psi_2 +\int_{\Omega} v_2\psi_2,
\end{array}
\end{equation}
for all $\psi_1, \psi_2 \in H^1(\Omega)$. Next, as $n \to \infty$, on the RHS of \eqref{weak-limit} we will use the Lebesgue dominated convergence theorem. Indeed, since $H^1(\Omega)\hookrightarrow L^2(\partial\Omega)$, we have $v_{i,n}\rightarrow v_i$ $i=1,2$ in $L^2(\partial\Omega)$, and from \eqref{eq:reminders} we get
\begin{equation}
\label{eq:limit-2}
\begin{array}{rcll}
    \lim_{n\rightarrow \infty}\lambda_n\scaleint{7ex}_{\hspace{-.3cm}\p\Om} \left(f_1^{\prime}(0)v_{2,n}+\dfrac{\mathcal{R}_1(u_{2,n})}{\|(u_{1,n},u_{2,n})\|}\right)\psi_1&=\underline{\lambda} \int_{\partial\Omega} f_1^{\prime}(0)v_2 \, \psi_1\\[.5cm]
    \lim_{n\rightarrow \infty}\lambda_n\scaleint{7ex}_{\hspace{-.3cm}\p\Om} \left(f_2^{\prime}(0)v_{1,n}+\dfrac{\mathcal{R}_2(u_{1,n})}{\|(u_{1,n},u_{2,n})\|}\right)\psi_2&=\underline{\lambda} \int_{\partial\Omega} f_2^{\prime}(0)v_1\psi_2\,.
\end{array}
\end{equation}
Therefore, \eqref{eq:limit-1} and \eqref{eq:limit-2} together imply that $(v_1,v_2)$ is a weak solution to the following problem
\begin{equation*}
 \left.
\begin{array}{rclll}
-\De v_{1}+v_{1} & = 0 \quad \mbox{in}\quad \Omega\,,&
&\frac{\partial v_{1}}{\partial \eta}&= \underline{\lambda}f_1'(0)v_2
\quad \mbox {on}\quad \partial\Omega\,;\\
-\Delta v_{2} +v_{2} &=  0 \quad \mbox{in}\quad \Omega\,, &
&\frac{\partial v_{2}}{\partial \eta}&=\underline{\lambda} f_2'(0)v_1 \quad \mbox{on}\quad \partial\Omega\,.
\end{array}
\right\} 
\end{equation*}

Moreover, 
\begin{equation}\label{z:w}
\begin{pmatrix}z_1\\z_2\end{pmatrix}=P^{-1}\begin{pmatrix}v_1\\v_2\end{pmatrix}
\end{equation}
satisfies 
\begin{align}
-\De z_{1}+z_{1} = 0 \quad \mbox{in}\quad \Omega\,,\qquad 
&\frac{\partial z_{1}}{\partial \eta}= \underline{\lambda}\sigma z_1
\quad \mbox {on}\quad \partial\Omega\,; \label{evp:1}\\
-\Delta z_{2} +z_{2} =  0 \quad \mbox{in}\quad \Omega\,, \qquad
&\frac{\partial z_{2}}{\partial \eta}+\underline{\lambda} \s z_2=0 \quad \mbox{on}\quad \partial\Omega\, \label{evp:2}
\end{align}
where $P$ is the invertible matrix associated with the Jordan canonical form presented in \eqref{eq:diagonal}.
Clearly, Problem \eqref{evp:2} is a homogeneous Robin boundary value problem that admits only the trivial solution, which implies $z_2 \equiv 0$, and the Steklov eigenvalue problem \eqref{evp:1} must satisfy $\underline{\lambda}=\mu_1/\sigma$ and $z_1=c\varphi_1$ for some $c \in \R$. Hence, $(z_1,z_2)=(c\varphi_1,0)$. Now, from \eqref{z:w}, we have
\begin{equation}
\label{eq:vi:phi}
\begin{pmatrix} v_1\\v_2\end{pmatrix}=  P\begin{pmatrix}
       \vf_1\\ 0
   \end{pmatrix}=\dfrac{1}{1+\z} \begin{pmatrix}1\\ \z\end{pmatrix} c\varphi_1, 
\qquad \mbox{ and }\quad  \underline{\la}=\mu_1/\s.  
\end{equation}
Next, note that $(v_1,v_2)=(\phi_1,\phi_2)$ and $(v_{1,n},v_{2,n})\to (v_1,v_2)$ in $(L^2(\Om))^2$, which implies $\|(v_{1,n},v_{2,n})-(v_1,v_2)\|_{(L^2(\Om))^2}\to 0.$ Additionally, due to $(v_{1,n},v_{2,n})\to (\phi_1,\phi_2)$ in $(C(\Omb))^2$, we immediately deduce that $\|(v_{1,n},v_{2,n})-(\phi_1,\phi_2)\|_{(L^2(\Om))^2}\le |\Om|\, \|(v_{1,n},v_{2,n})-(\phi_1,\phi_2)\|_{(C(\overline{\Om}))^2}\to 0$ and the uniqueness of the limit in $(L^2(\Om))^2$ implies that $(\phi_1,\phi_2)=(v_1,v_2)$ with $\|(v_1,v_2)\|_{(C(\overline{\Om}))^2}$ $=1$. Thus we have, $\|z_1\|_{C(\overline{\Omega})}=1$, 
and by the definition of the norm (see \eqref{norm}) and \eqref{eq:vi:phi}, we obtain $c=1$, and 
\begin{equation}
\label{z1:z2}
\begin{pmatrix}
\dfrac{u_{1,n}}{\|(u_{1,n}, u_{2,n})\|}\\ 
\dfrac{u_{2,n}}{\|(u_{1,n}, u_{2,n})\|}
\end{pmatrix}\to P\begin{pmatrix}
       \vf_1\\ 0
   \end{pmatrix}= \dfrac{1}{1+\z} \begin{pmatrix}1\\ \z\end{pmatrix} \varphi_1 =
   \begin{pmatrix} \phi_1\\
   \phi_2\\
   \end{pmatrix}
   \qq{as}n\rightarrow \infty \,.
\end{equation}
This completes the proof.
\end{proof}
The next theorem is a version of Crandall and Rabinowitz's local bifurcation result and of Rabinowitz's global bifurcation result, see \cite{C-R}  and \cite{R71}, demonstrating the fact that the bifurcation point $(\mu_0, (0,0))$ is in fact unique and the connected component bifurcating from the trivial solution at the point $(\mu_0, (0,0))$ is unbounded.

\begin{thm}
\label{thm:bif:point}
If  $f_i \in C^1([0,\infty))$ satisfy $\eqref{H0}$, then there exists a connected component $\mathscr{C}^+ \subset \Sigma$ of positive weak solution to \eqref{pde} emanating from the trivial solution at $(\mu_0,(0,0))\in \mathbb{R}\times \big(C(\overline{\Omega})\big)^2$ with $\mu_0$ defined in \eqref{def:mu0}. Moreover, $\mathscr{C}^+$ is unbounded in $\mathbb{R} \times \big(C(\overline{\Omega})\big)^2$.
\end{thm}
\begin{proof}
Utilizing hypothesis \eqref{H0}, let us define $\mathscr{F}: \mathbb{R} \times \big(C({\p\Omega})\big)^2 \to \big(C({\p\Omega})\big)^2$ as follows. 
\begin{equation*}
    \mathscr{F}(\lambda, (u_1, u_2))=  \begin{pmatrix} 
     u_1\\
    u_2
    \end{pmatrix} -
    \begin{pmatrix} 
    S(\lambda f_1(u_2))\\
    S(\lambda f_2 (u_1))
    \end{pmatrix}
    = \begin{pmatrix} 
     u_1\\
    u_2
    \end{pmatrix} -
    S \left( \lambda A \begin{pmatrix} 
     u_1\\
    u_2
    \end{pmatrix}+\lambda
    \begin{pmatrix} 
     \mathcal{R_1}(u_2)\\
    \mathcal{R_2}(u_1)
    \end{pmatrix}\right),
\end{equation*}
where the matrix $A$  is defined in \eqref{eq:matrix-A} and the operator $S$ is defined in \eqref{D:to:N}.
Primarily, we seek nontrivial solutions $\big(\lambda,(u_1,u_2)\big)$ of $\mathscr{F}\big(\lambda,(u_1,u_2)\big)=(0, 0)^t,$ employing the Crandall-Rabinowitz Bifurcation Theorem (see \cite[Theorem 1.7]{C-R}), where $(a, b)^t$ denotes the transpose of the vector 
$\bigg(\begin{array}{cc}
     a\\
    b
    \end{array}\bigg) .$ Differentiating $\mathscr{F}$ with respect to $(u_1, u_2)$ and evaluating it at $(\la,(u_1,u_2))=(\mu_0,(0,0))$, we get $L_0(v_1,v_2):=D_{(v_1,v_2)}\mathscr{F}(\mu_0,(0,0))(v_1,v_2)=(v_1,v_2)^t-\mu_0 S\big(A(v_1,v_2)^t\big).$ Then, by  differentiating  $D_{(u_1,u_2)}\mathscr{F}$ again w.r.t. $\lambda$, we deduce that $L_1(v_1,v_2):=D^2_{\lambda\,\, (v_1,v_2)}\mathscr{F}(\mu_0,(0,0))(v_1,v_2)=-S\big(A(v_1,v_2)^t\big).$
Observe that Proposition \ref{prop:limit} implies $ker(L_0)=span\{(\phi_{1},\phi_{2})\}$, where $(\phi_1, \phi_2)$ are defined in \eqref{phi1} and \eqref{phi2}. Hence $dim(ker(L_0))=1$. Now we demonstrate that the transversality condition holds, that is $L_1(ker(L_0)) \not\subset Range(L_0)$. 
In contrast, assume that there exists a $\gamma \in \mathbb{R}\backslash \{0\}$ with $\gamma (\phi_1,\phi_2) \in ker(L_0)$ such that $\gamma L_1 (\phi_1,\phi_2) \in Range (L_0)$. This holds if and only if there exists $(w_1,w_2)$ such that $\gamma L_1 (\phi_1,\phi_2)  =L_0(w_1,w_2)$, equivalently, $(w_1,w_2)^t  = S\Big( \mu_0A(w_1,w_2)^t-\gamma A(\phi_1,\phi_2)^t\Big)\,.$
Then $(z_1,z_2)=P^{-1}(w_1,w_2)^t$ satisfies $(z_1,z_2)^t = S\Big( \mu_0 J (z_1,z_2)^t-\gamma JP^{-1}(\phi_1,\phi_2)^t\Big)\,,$
where $J$ is given by \eqref{eq:diagonal}.
Now, the definitions of $\mu_0$ given by \eqref{bif:pt}, $\s$ given by \eqref{s+-},  and
$(\phi_1, \phi_2)$ by \eqref{phi1}--\eqref{phi2}, implies that $(z_1,z_2)$ satisfies
\begin{equation}
\label{pde_z1}
 \left.
\begin{array}{rcllll}
-\De z_{1}+z_{1} & = 0 \quad \mbox{in}\quad \Omega\,,&
&\frac{\partial z_{1}}{\partial \eta}&=\mu_1z_1-\gamma \s \varphi_1 
& \quad \mbox {on}\quad \partial\Omega\,,\\
-\Delta z_{2} +z_{2} &=  0 \quad \mbox{in}\quad \Omega\,, &
&\frac{\partial z_{2}}{\partial \eta}&=-\mu_1z_2 & \quad \mbox{on}\quad \partial\Omega\,,
\end{array}
\right\} 
\end{equation}
where $\varphi_1$ is the principal eigenfunction. 
Moreover, using the weak formulation for the first equation in \eqref{pde_z1} with the test function $\varphi_1$, we obtain 
\begin{equation*}
\mu_1\int_{\partial\Omega}z_1\varphi_1 = \int_\Omega \nabla z_1 \nabla \varphi_1+z_1\varphi_1 = \mu_1\int_{\partial \Omega} z_1\varphi_1-\gamma  
\s  \int_{\partial\Omega} \varphi_1^2\notag,
\end{equation*}
which implies
$   \gamma 
   \s\int_{\partial\Omega} \varphi_1^2=0.
$
Consequently, $\gamma=0$ since $\s\ne0$ by hypothesis \eqref{H0}, which is a contradiction. 
Hence $(\mu_0,(0,0))$ is a bifurcation point, and by Rabinowitz theorem (see \cite{R71}, Thm.~1.3) there exists a connected component $\mathscr{C}^+ \subset \Sigma\,$  of positive weak solutions of \eqref{pde} emanating from the trivial solution at $(\mu_0,(0,0))\in \mathbb{R}\times \big(C(\overline{\Omega})\big)^2$ where the branch $\mathscr{C}^+$ either meets another bifurcation point from the trivial solution, or it is unbounded in $\mathbb{R}\times \big(C(\overline{\Omega})\big)^2$. Since $f_i \geq 0$ satisfies $\eqref{H0}$ for $i=1,2$ it follows from   \cite[Lem.~2.1(iv) \& Prop.~2.5]{BCDMP_nonlbc} that the branch contains only positive solutions. From the Crandall-Rabinowitz Theorem (see \cite{C-R}), $\mathscr{C}^+$ can neither meet another bifurcation point from zero, nor can meet $(\mu_0,(0,0))$ again. Thus, the branch $\mathscr{C}^+$ 
 is unbounded in $\mathbb{R}\times \big(C(\overline{\Omega})\big)^2$.
\end{proof}

\subsection{Proof of Theorem~\ref{thm-2}}
Having established the preceding  Proposition \ref{prop:limit} and Theorem \ref{thm:bif:point}, we are now in a position to prove Theorem \ref{thm-2}. The proof proceeds in several steps, as detailed below.

\textbf{Step 1:} {\it Existence of an unbounded connected component $\mathscr{C}^+$ of positive weak solutions of \eqref{pde}.} 

By Theorem \ref{thm:bif:point}, there exists a connected component $\mathscr{C}^+$ of positive weak solutions of \eqref{pde} bifurcating from the trivial solution at the bifurcation point $\left(\mu_0,(0,0)\right)$ and that $\mathscr{C}^+$ is unbounded in $\mathbb{R}\times \big(C(\overline{\Omega})\big)^2$, where $\mu_0$ is as defined in \eqref{def:mu0}. 

\textbf{Step 2:} {\it Next, we prove nonexistence of positive solutions of \eqref{pde} for $\lambda > \frac{\mu_1}{K}$, where $K$ is given in \eqref{K}.}  

Suppose in contrast that there exists a positive solution $(u_1, u_2)$  corresponding to $\lambda > \frac{\mu_1}{K}$. Then, taking $\varphi_1>0$ as a test function in the weak formulation  \eqref{eq:weak-sol}, the first equation satisfies
\begin{align}
 0 &= 
\lambda\int_{\partial\Omega} f_1(u_2)\varphi_1-\int_{\Omega}\left[ \nabla u_1\nabla \varphi_1+u_1\varphi_1\right]
 \geq \lambda K \int_{\partial\Omega} u_2\varphi_1 - \mu_1 \int_{\partial\Omega} u_1\varphi_1 
>\mu_1\int_{\partial\Omega} (u_2-u_1)\varphi_1\,.
\label{eq_1}
\end{align}
On the other hand, using the second equation of  \eqref{eq:weak-sol}, we get 
\begin{equation*}
    \int_{\partial\Omega} (u_2-u_1) \varphi_1> 0\,,
\end{equation*}
which contradicts \eqref{eq_1}, establishing the claim.\\

\textbf{Step 3:}
{\it We show that $\mathscr{C}^+$ contains a unique bifurcation point from infinity at $\lambda=0$.} 
\par It follows from  {\it Step~1} that $\mathscr{C}^+$ is unbounded in $\mathbb{R}\times \big(C(\overline{\Omega})\big)^2$ and  {\it Step 2} implies that $\mathscr{C}^+$ is bounded in the $\lambda$-direction.  
Hence, 
there exists a sequence  $(\lambda_n, (u_{1,n},u_{2,n})) \in \mathscr{C}^+$ such that  $\lambda_n \in \left[0, \frac{\mu_1}{K}\right]$ and $\|(u_{1,n},u_{2,n})\|\to \infty$.
By choosing a subsequence if necessary, $\lambda_n\to \widetilde{\lambda}$ and $\|(u_{1,n},u_{2,n})\|\to \infty$.  We claim that $\widetilde{\lambda}=0$.

\par Assume to the contrary that $\widetilde{\lambda}> 0$ and for $a_0>0,$ let  $[a_0,b_0]$ be  any  fixed compact interval with $\widetilde{\lambda}\in(a_0,b_0)$. By the uniform \emph{a priori} bound result,  Proposition \ref{th:Rossi}, 
for any $\lambda\in[a_0,b_0]$,   there exists a uniform constant $M=M(a_0,b_0) > 0$ such that  for every $(\lambda,(w_1,w_2))$ with $\lambda\in[a_0,b_0]$ and $(w_1,w_2)$  a positive  weak solution  of the re-scaled problem \eqref{pde_rscld},  we have $\|(w_1,w_2)\|\leq M\,.$ Now, by \eqref{eq:re-scaling} it follows that for any $\lambda>0$, $(u_1,u_2)$ is a positive weak solution to \eqref{pde} if and only if $(w_1,w_2)=(\lambda^{\theta_1}u_1,\lambda^{\theta_2}u_2)$ is a weak solution to  \eqref{pde_rscld}, where $\theta_1,\, \theta_2$ are given by \eqref{theta_eq}.
Consequently, 
\begin{equation}\label{eq:bound-norm-u}
\|(u_1,u_2)\|_{C(\overline{\Omega})^2}
\leq \max\{\lambda^{-\theta_1},\lambda^{-\theta_2}\}M\leq \max\{a_0^{-\theta_1},a_0^{-\theta_2}\}M=:M'\mbox{ for  any } \lambda \in[a_0,b_0],
\end{equation} 
which contradicts the fact that $\|(u_{1,n},u_{2,n})\|\to \infty$ as $\lambda_n\to \widetilde{\lambda}>0$. Hence, we can conclude that $\widetilde{\lambda}=0$ and $\mathscr{C}^+$ contains a unique bifurcation point from infinity at $\lambda=0$ and \eqref{unbdd:zero} holds necessarily. With this final step, proof of Theorem \ref{thm-2} is complete.

\section{Multiplicity Result: Proof of Theorem \ref{thm-3}}
\label{sec:multiplicity}
To present the multiplicity of positive solutions, we first need to determine the direction of bifurcation at the bifurcation point. This direction, either to the left or to the right, depends on the sign of $\underline{\mathcal{R}}_0$ and $\overline{\mathcal{R}}_0$ introduced in \eqref{eq:R-zero}. To this end, we begin by proving the following lemma. We then state a theorem that characterizes the bifurcation direction. Once the lemma is established, the proof of the theorem follows immediately. 

\begin{lem}\label{lemma-inf-sup}
 Suppose that $f_1, f_2\in C^1([0,\infty))$ satisfy the hypothesis $\eqref{H0}$. Let $\{(u_{1,n}, u_{2,n})\}$ be a sequence of positive weak solutions  of \eqref{pde} corresponding to the parameter $\lambda_n$ such that $\lambda_n\rightarrow \mu_0
 $ and $\|(u_{1,n}, u_{2,n})\|\rightarrow 0$, where $\mu_0$ is defined in \eqref{bif:pt}. Then, 
\begin{align*}
\frac{\mu_0}{\s}\ \underline{\mathcal{R}}_0 \ \dfrac{\int_{\partial\Omega} \varphi_1^{1+\nu}}{\int_{\partial\Omega} \varphi_1^2}
&\leq \liminf_{n\rightarrow \infty} \dfrac{\mu_0-\lambda_n}{\|(u_{1,n}, u_{2,n})\|^{\nu-1}}\nonumber\\
&\leq \limsup_{n\rightarrow \infty} \dfrac{\mu_0-\lambda_n}{\|(u_{1,n}, u_{2,n})\|^{\nu-1}}\leq \frac{\mu_0}{\s}\ \overline{\mathcal{R}}_0 \ \dfrac{\int_{\partial\Omega} \varphi_1^{1+\nu}}{\int_{\partial\Omega} \varphi_1^2},
\end{align*}
where $\s,\ \underline{\mathcal{R}}_0$ and $\overline{\mathcal{R}}_0$ were defined in  \eqref{s+-} and \eqref{eq:R-zero} respectively.
\end{lem}
\begin{proof}
It is readily seen that  using the relationship \eqref{relation:w:u}, the system \eqref{eq:reduction}-\eqref{eq:reduction-boundary}  is equivalent to the system \eqref{pde} under Hypothesis \eqref{H0}. Thereafter, by using the weak formulation of \eqref{eq:reduction}-\eqref{eq:reduction-boundary}, with $\varphi_1$ as a test function, and hypothesis $\eqref{H0}$, we get
that the first component of  the weak formulation of \eqref{eq:reduction}-\eqref{eq:reduction-boundary} 
satisfies
\begin{equation}
	\label{eq:bound-1}
	\begin{array}{rcll}
\dfrac{\mu_0-\lambda_n}{\|(u_{1,n},u_{2,n})\|^{\nu-1}}
\displaystyle\scaleint{7ex}_{\hspace{-.3cm}\p\Om} \frac{w_{1,n}}{\|(u_{1,n},u_{2,n})\|}\varphi_1
=\dfrac{\lambda_n}{\sigma} \displaystyle\scaleint{7ex}_{\hspace{-.3cm}\p\Om}\frac{\mathcal{R}_1^*(w_{1,n},w_{2,n})}{\|(u_{1,n},u_{2,n})\|^{\nu}}\varphi_1\,,
\end{array}
\end{equation}
where $\s$, $\mu_0$ and $\mathcal{R}_1^*$ are given in \eqref{s+-}, \eqref{bif:pt} and \eqref{R*}, respectively.
Applying Fatou's lemma to the right-hand side of \eqref{eq:bound-1} and applying \eqref{R*}, we obtain 
\begin{align}
\label{eq:bound-2}
& \liminf_{n\rightarrow \infty} \displaystyle\scaleint{7ex}_{\hspace{-.3cm}\p\Om}\frac{\mathcal{R}_1^*(w_{1,n},w_{2,n})}{\|(u_{1,n},u_{2,n})\|^{\nu}}\varphi_1 \notag \\
&=\frac{1+\z}{2}\liminf_{n\rightarrow \infty}\displaystyle\scaleint{7ex}_{\hspace{-.3cm}\p\Om} \left[\dfrac{\mathcal{R}_1(u_{2,n})}{u_{2,n}^{\nu}}\left(\dfrac{u_{2,n}}{\|(u_{1,n},u_{2,n})\|}\right)^{\nu}+ \dfrac{1}{\z}\dfrac{\mathcal{R}_2(u_{1,n})}{u_{1,n}^{\nu}}\left(\dfrac{u_{1,n}}{\|(u_{1,n},u_{2,n})\|}\right)^{\nu} \right] \varphi_1 \nonumber\\
&\geq \frac{1+\z}{2}\displaystyle\scaleint{7ex}_{\hspace{-.3cm}\p\Om}  \liminf_{n\rightarrow \infty} \left[\dfrac{\mathcal{R}_1(u_{2,n})}{u_{2,n}^{\nu}}\left(\dfrac{u_{2,n}}{\|(u_{1,n},u_{2,n})\|}\right)^{\nu} +\dfrac{1}{\z}\dfrac{\mathcal{R}_2(u_{1,n})}{u_{1,n}^{\nu}}\left(\dfrac{u_{1,n}}{\|(u_{1,n},u_{2,n})\|}\right)^{\nu} \right]\varphi_1  \nonumber\\
&\geq \frac{1}{2}\displaystyle\scaleint{7ex}_{\hspace{-.3cm}\p\Om}   \left[\left(\dfrac{\zeta}{1+\zeta}\right)^{\nu-1}\underline{\mathcal{R}}_1+    
    \left(\dfrac{1}{1+\zeta}\right)^{\nu-1}
     \underline{\mathcal{R}}_2\right]\, \varphi_1^{1+\nu}= \underline{\mathcal{R}}_0\int_{\p\Om}  \, \varphi_1^{1+\nu},  
\end{align}
Note that, here we applied the limit outlined in Equation \eqref{z1:z2} of Proposition \ref{prop:limit} and the definitions of $\mathcal{R_i}$'s are given \eqref{eq:def-Ris}-\eqref{eq:R-zero}.
Now, combining \eqref{eq:bound-1}, \eqref{eq:bound-2}, and \eqref{bif:pt}
we get 
\begin{equation*}
    \liminf_{n\rightarrow \infty} \dfrac{\mu_0-\lambda_n}{\|(u_{1,n}, u_{2,n})\|^{\nu-1}}\geq \frac{\mu_0}{\s}\, \underline{\mathcal{R}}_0\, \dfrac{\int_{\partial\Omega} \varphi_1^{1+\nu}}{\int_{\partial\Omega} \varphi_1^2},
\end{equation*}
Analogously, one can deduce
\begin{equation*}
    \limsup_{n\rightarrow \infty} \dfrac{\mu_0-\lambda_n}{\|(u_{1,n}, u_{2,n})\|^{\nu-1}}\leq \frac{\mu_0}{\s}\, \overline{\mathcal{R}}_0 \, \dfrac{\int_{\partial\Omega} \varphi_1^{1+\nu}}{\int_{\partial\Omega} \varphi_1^2}
\end{equation*}
Finally, $\liminf_{n\rightarrow \infty} \dfrac{\mu_0-\lambda_n}{\|(u_{1,n}, u_{2,n})\|^{\nu-1}}\nonumber
\leq \limsup_{n\rightarrow \infty} \dfrac{\mu_0-\lambda_n}{\|(u_{1,n}, u_{2,n})\|^{\nu-1}}$, follows from the definition of limit supremum and limit infimum, which completes the proof.
\end{proof}
\begin{thm}
\textbf{(Direction of bifurcation from the trivial solution)}\\
Assume that the nonlinearities $f_1$ and $f_2$ satisfy Hypothesis $\eqref{H0}$. Then, the following holds
\begin{enumerate}
    \item [(i)] \textbf{(Bifurcation to the left)} If $\underline{\mathcal{R}}_0>0$, then the bifurcation of positive weak solutions from the trivial solution at $\lambda=\mu_0$ is to the left.  That is,  $\lambda<\mu_0$ for $\lambda$ in a neighborhood of $\mu_0$.
    \item [(ii)] \textbf{(Bifurcation to the right)} If $\overline{\mathcal{R}}_0<0$, then the bifurcation of positive weak solutions from the trivial solution at $\lambda=\mu_0$ is to the right. That is,  $\lambda>\mu_0$ for $\lambda$ in a neighborhood of $\mu_0$.
\end{enumerate}
\end{thm}
\begin{proof}
This proof is an immediate consequence of Lemma \ref{lemma-inf-sup}.
\end{proof}

We now proceed to prove Theorem \ref{thm-3}, which establishes the multiplicity of positive solutions. As indicated in the preceding theorem, when \( \overline{\mathcal{R}}_0 < 0 \), the branch \( \mathscr{C}^+ \) bifurcates to the right. In that scenario, we first apply degree theory to prove the existence of a first positive solution. To obtain a second solution, we then employ the sub- and supersolution method, where the trivial  solution serves as a subsolution and the previously obtained solution (via degree theory) acts as a supersolution. Finally, we also establish the existence of positive solutions both at the bifurcation point and at $\overline{\lambda}$.

\subsection{Proof of Theorem~\ref{thm-3}}

To begin with, observe that, the connectedness of $\mathscr{C}^+$ proved in Theorem~\ref{thm-2} guarantees existence of a positive solution for any $\lambda \in \left(0, \mu_0\right)$. 
Additionally, the nonexistence result of positive solutions for $\lambda>\frac{\mu_1}{K}$ guarantees, $\overline{\lambda}$ as defined in \eqref{la:up} is well defined. The proof will be carried out following $3$ steps. Since the connected component $\mathscr{C}^+$ bifurcates to the right of the bifurcation point, we now show that there exist two solutions of \eqref{pde} for each $\la \in \left(\mu_0, \overline{\la }\right)$. We first fix $\la \in \left(\mu_0, \overline{\la }\right)$ and then fix $\la_0 \in (\la, \overline{\la})$.

\subsubsection{Step~1: Existence of a positive solution $(u_1 ^*,u_2^*)$ for each  $\lambda\in \left(\mu_0, \overline{\lambda}\right)$ via degree theory}

Note that, by Proposition~\ref{th:Rossi}, there exists $M>0$ such that $\|(u_1,u_2)\|_{C(\overline{\Omega})^2}\leq M$ for any solution $(u_1,u_2)$ of \eqref{pde}.
Now, let $\theta\in [0,1]$ and  for a given $(u_1,u_2)\in \big(C(\partial \Omega)\big)^2$, let $(v_1,v_2)$ be a solution to
\begin{equation}\label{eq:defintion-T}
\left.
\begin{array}{rclll}
-\De v_1+v_1 & = 0 \quad \mbox{in}\quad \Omega\,,&
&\frac{\partial v_1}{\partial \eta}&=f_{1,\theta}(u_2)
\quad \mbox {on}\quad \partial\Omega\,;\\
-\Delta v_2 +v_2 &=  0 \quad \mbox{in}\quad \Omega\,, &
&\frac{\partial v_2}{\partial \eta} &=  f_{2,\theta}(u_1)\quad \mbox {on}\quad \partial\Omega\,;
\end{array}
\right\} 
\end{equation}
with 
$$
f_{i,\theta}(u_j)=\theta\, \lambda f_i(u_j)+(1-\theta)(\beta u_j^+ +1) \qq{for} i\ne j\,, 
$$
where $u^+ = \text{max}\{ u,0\}$ and $\beta>\mu_1\max\left\{1, \dfrac{b}{a}, \dfrac{a}{b}\right\}$ with $a=\sqrt{f_1'(0)}$, $b=\sqrt{f_2'(0)}$.
Then the fixed point operator $T_{\theta}\colon \big(C(\partial \Omega)\big)^2\rightarrow \big(C(\partial \Omega)\big)^2$, associated to \eqref{eq:defintion-T}, is given by 
\begin{align}\label{T:te}
  T_{\theta}(u_1,u_2)&:=(S\circ F_{\theta})(u_1,u_2)=(\Ga v_1,\Ga  v_2)\,,
\end{align}
where $F_{\theta}=(f_{1,\theta}, f_{2,\theta})$. Clearly, $T_{\theta}$ is a compact operator, as discussed in Subsection~\ref{subsec:3.2} with $\tilde F$ instead of $F_{\theta}$.  Moreover, by \eqref{eq:fixed-point} one can see that a fixed point of the operator $T_1$ is basically a solution to Problem \eqref{pde}.

For each \( \theta \in [0,1] \), let \( (u_{1,\theta}, u_{2,\theta}) \) denote a fixed point of \eqref{T:te}. Proposition~\ref{th:Rossi} implies there exists uniform $M''>0$  such that 
\begin{equation}\label{M''}
\|(u_{1,\te}, u_{2,\te}) \|_{(C(\overline{\Omega}))^2}\le M''  \qq{for all}  \te\in [0,1].
\end{equation}
Additionally, Proposition~\ref{max_2} ensures that there exists some $\eps>0$  small enough such that
\begin{equation}\label{u:te:ab}
 u_{1,\te}(x)>a\e  \varphi_1(x) \qq{and}  u_{2,\te}(x)>b\e  \varphi_1(x),\quad \forall \te\in[0,1],\quad \forall x\in\Omb\,.
\end{equation}
Indeed, from \eqref{eq:defintion-T}  and \eqref{steklov}, we obtain 
\begin{align*}
   -\Delta(u_{1,\te}-a\e  \varphi_1)+(u_{1,\te}-a\e  \varphi_1)&=0 \quad \text{ in }\Omega,\\
   -\Delta(u_{2,\te}-b\e  \varphi_1)+(u_{2,\te}-b\e  \varphi_1)&=0 \quad \text{ in }\Omega\,
\end{align*}
and then, using  $\lambda_0>\frac{\mu_1}{ab}$, 
and Hypothesis $\eqref{H0}$, we obtain for all  $\theta \in[0,1]$
\begin{align*}
    \dfrac{\partial }{\partial\eta}(u_{1,\te}-a\e  \varphi_1)&=\lambda_0 f_{1,\te}(u_{2,\te})-a\e  \mu_1 \varphi_1> \lambda_0 \left(f_{1,\te}(u_{2,\te})-\e  a^2 b \varphi_1\right)> 0\quad \text{ on }\partial\Omega,
\end{align*}
\begin{align*}
    \dfrac{\partial }{\partial\eta}(u_{2,\te}-b\e  \varphi_1)&=\lambda_0 f_{2,\te}(u_{1,\te})-b\e  \mu_1 \varphi_1> \lambda_0 \left(f_{2,\te}(u_{1,\te})-\e  b^2 a \varphi_1\right)> 0 \quad \text{ on }\partial\Omega,
\end{align*}
for $\e $ sufficiently small. 
Thereafter, we define
\begin{equation*}
   Y:=\{(v_1,v_2)\in \big(C(\overline{\Omega})\big)^2 \colon \|(v_1,v_2)\|< 2M'',\ v_1>
   a \e \varphi_1,\ v_2>b\e  \varphi_1 \mbox{ in }\overline{\Omega}\},
\end{equation*}
where $M''$ is as defined in \eqref{M''} with $a_0= \frac{\mu_1}{ab}$ and $b_0=\overline{\lambda}$ and 
\begin{equation}
\label{eq:DefZ}
Z:=\{(v_1,v_2)\in Y\colon v_1<u^0_{1}, \ v_2<u^0_{2} \text{ in }\overline{\Omega}\},
\end{equation}
where $(u^0_{1},u^0_{2})$ is a solution to \eqref{pde} for $\lambda=\lambda_0$, ensured by the definition of $\overline{\la}$, see \eqref{la:up}. Moreover, we can take $\e >0$  sufficiently small in order to guarantee that $(u_{1}^{0},u_{2}^{0})\in Y$.
\vspace{0.1in}

\noindent \textbf{Claim 1:} $\deg(I-T_1,Y,0)=0$. 
\vspace{0.1in}
\begin{proof}
\noindent First step towards the proof of the claim is to justify that $\deg(I-T_{\theta},Y,0)$ for any $\te\in[0,1]$ is well-defined. Observe that, if $(u_1,u_2)\in \partial Y$, then either $\|(u_1,u_2)\|_{(C(\overline{\Omega}))^2}=2M''$, or there exists a point $x_0\in\Omb$ such that either $u_1(x_0)=a\e  \varphi_1(x_0)$, or $u_2(x_0)=b\e  \varphi_1(x_0)$.  
Furthermore, observe that if $(u_{1,\te},u_{2,\te})$ is a fixed point of the equation $(u_{1,\te},u_{2,\te})=T_{\theta}(u_{1,\te},u_{2,\te})$, then by \eqref{M''} and \eqref{u:te:ab}, $(u_{1,\theta},u_{2,\theta})\not\in\p Y$ and $\deg(I-T_{\theta},Y,0)$ is independent of $\te.$
Next, we demonstrate that $(v_1,v_2)\ne T_0((v_1,v_2))$ for any $(v_1,v_2)\in Y$. For our purposes, suppose that $(v_1,v_2)= T_0(v_1,v_2)$. Then for $\te=0$, System \eqref{eq:defintion-T}  reduces to
\begin{equation}\label{eq:reduction-1}
	\left.
	\begin{array}{rclll}
-\De v_1+v_1 & = 0 \quad \mbox{in}\quad \Omega\,,&
&\frac{\partial v_1}{\partial \eta}&=\beta v_2^++1\quad \mbox {on}\quad \partial\Omega\,;\\
-\Delta v_2 +v_2 &=  0 \quad \mbox{in}\quad \Omega\,, &
&\frac{\partial u_2}{\partial \eta}&=\beta v_1^++1\quad \mbox {on}\quad \partial\Omega\,.
\end{array}
	\right\} 
\end{equation}
Thereafter, taking $\varphi_1>0$, the eigenfunction associated to the first Steklov eigenvalue $\mu_1$, as the test function in the weak formulation of \eqref{eq:reduction-1},   we obtain
\begin{align}\label{eq:inequality-1}
&\mu_1\int_{\partial\Omega}v_1\varphi_1 =\int_{\Omega} \nabla v_1 \nabla \varphi_1+\int_{\Omega} v_1\varphi_1=\int_{\partial\Omega} (\beta v_2^++1)\varphi_1
\end{align}
and 
\begin{align}\label{eq:inequality-2}
&\mu_1\int_{\p\Om}v_2\vf_1 =\int_{\Omega} \nabla v_2 \nabla \varphi_1+\int_{\Omega} v_2\varphi_1=\int_{\partial\Omega} (\beta v_1^++1)\varphi_1.
\end{align}
Subsequently, adding Equations \eqref{eq:inequality-1} and \eqref{eq:inequality-2}, we readily deduce, 
$\mu_1\int_{\p\Om}(v_1+v_2)\vf_1 =\int_{\p\Om} (\be (v_1^++v_2^+)+2)\vf_1$.
Moreover, due to the fact that $\mu_1>0$ and $\varphi_1>0$ on $\partial \Omega$, and that $\beta >\mu_1$, we achieve a contradiction.
Consequently, $(v_1,v_2)\ne T_0((v_1,v_2))$ for all $(v_1,v_2)\in Y$. Finally, using $\theta\in[0,1]$ as homotopy parameter, we get
\begin{equation}\label{eq:degree-1}
\deg(I-T_1,Y,0)=\deg(I-T_{\theta},Y,0)=\deg(I-T_0,Y,0)=0,
\end{equation}
which completes the proof.
\end{proof}
\vspace{0.1in}

\noindent \textbf{Claim 2:} $\deg(I-T_1,Z,0)=1$.
\vspace{0.1in}
\begin{proof}
\noindent We begin by fixing $(\psi_{1},\psi_{2})\in Z$ for any $\gamma\in[0,1]$. Then for any $(u_1,u_2)\in  Z$,
let us consider the problem
 \begin{equation}\label{eq:reduction-2}
	\left.
	\begin{array}{rclll}
-\De v_1+v_1 & = 0 \quad \mbox{in}\quad \Omega\,,&
&\frac{\partial v_1}{\partial \eta}&=\gamma\lambda f_1(u_2)+(1-\gamma)\psi_{1}\quad \mbox {on}\quad \partial\Omega\,;\\
-\Delta v_2 +v_2 &=  0 \quad \mbox{in}\quad \Omega\,, &
&\frac{\partial v_2}{\partial \eta}&=\gamma\lambda f_2(u_1)+(1-\gamma)\psi_{2}\quad \mbox {on}\quad \partial\Omega\,.
\end{array}
	\right\} 
\end{equation}
Note that, we can rewrite Eq.\eqref{eq:reduction-2} in the following way 
\begin{equation}\label{eq:sum-map}
 (v_1,v_2)=\left(\gamma T_1+(1-\gamma)\widetilde{T}\right)(u_1,u_2),   
\end{equation}
where $\widetilde{T}$ is a map that sends every element of $Z$ into a fixed $(\psi_{1},\psi_{2})\in Z$, that is, 
\begin{equation*}
\widetilde{T}(u_1,u_2)=(\psi_{1}, \psi_{2}),  \qq{for all} (u_1,u_2)\in Z. 
\end{equation*}
It is clear that when $\gamma = 1$, eq. \eqref{eq:reduction-2} is equivalent to  eq. \eqref{eq:defintion-T} with $\theta=1$. Next, we show that $\deg(I-(\gamma T_1+(1-\gamma)\widetilde{T}),Z,0)$ is well-defined and independent of $\gamma\in [0,1]$. For that, we prove that if  $(u_1,u_2) \in \overline{Z}$ (see \eqref{eq:DefZ}) then $(w_1,w_2)=T_1(u_1,u_2) \in Z$. In fact, $-\Delta w_i+w_i=0$ in $\Omega$ for $i=1,2$. Therefore using the definition of $T_1$ (see \eqref{T:te}), and the fact that $f_i$ is nondecreasing, $\la<\la_0$ and the definition of $(u_1^0, u_2^0)$, we  obtain 
\begin{equation*}
\begin{array}{rclll}
\frac{\partial w_1}{\partial \eta}&=\lambda f_1(u_2)&
\le\lambda f_1(u_{2}^{0})<\lambda_0f_1(u_{2}^{0})=\frac{\partial u_{1}^{0}}{\partial \eta},\\
\frac{\partial w_2}{\partial \eta}&=\lambda f_2(u_1)&
\le\lambda f_2(u_{1}^{0})<\lambda_0f_2(u_{1}^{0})=\frac{\partial u_{2}^{0}}{\partial \eta}.
\end{array}
\end{equation*}
Moreover,  using   $\eqref{H0}$, we obtain
\begin{equation*}
\begin{array}{rclll}
\frac{\partial w_1}{\partial \eta}&=\lambda f_1(u_2)&
>a\e\frac{\partial \vf_1}{\partial \eta},\\
\frac{\partial w_2}{\partial \eta}&=\lambda f_2(u_1)&
>b\e \frac{\partial \vf_1}{\partial \eta}.
\end{array}
\end{equation*}
The comparison principle described in Proposition \ref{max_2} implies $a\e\vf_1<w_1<u_{1}^{0}$ and $b\e\vf_1<w_2<u_{2}^{0}$, hence $T_1(u_1,u_2)\in Z$. It clearly follows from the convexity of $Z$, since $(\psi_{1},\psi_{2})\in Z$, the RHS of \eqref{eq:sum-map} belongs to $Z$ for any $(u_1, u_2) \in  \overline{Z}$ and for all $\gamma\in [0,1]$. Simultaneously, $Z$ is being an open set implies that there are no fixed points of \eqref{eq:sum-map} in $\p Z$, in other words, the degree $\deg(I-(\gamma T_1+(1-\gamma)\widetilde{T}),Z,0)$ is well-defined and independent of $\gamma\in [0,1]$.
Therefore, for $(\psi_{1},\psi_{2})\in Z$, we have
\begin{equation}\label{eq:degree-2}
    \deg(I-T_1,Z,0)=\deg(I-\widetilde{T},Z,0)=\deg(I,Z,(\psi_{1},\psi_{2}))=1.
\end{equation}
At the end, combining \eqref{eq:degree-1} and \eqref{eq:degree-2}, we get $\deg(I-T_1,Y\setminus \overline{Z},0)=-1$ and there exists a positive solution $(u_1 ^*,u_2^*)\in Y\setminus \overline{Z}$ of \eqref{pde} corresponding to $\lambda$. 
\end{proof}

\subsubsection{Step~2: Existence of a second positive solution for each  $\lambda\in \left(\mu_0, \overline{\la}\right)$ via sub and super solution theory}

We employ \cite[Theorem 1.4 ]{BLM_maxmin} in order to construct a second positive solution distinct from $(u_1 ^*,u_2^*)$. We note that the monotonicity of $f_i$'s and the regularity result in Theorem \ref{th:bootstrap} imply that $f_i$'s satisfy condition (A4) in \cite[Theorem 1.4 ]{BLM_maxmin}.
We claim that  $(\underline{u}_1,\underline{u}_2)=(\e a\varphi_1,\e b\varphi_1)$ is a subsolution to \eqref{pde} for $\e >0$ small enough where $a= \sqrt{f_1'(0)}$, $b=\sqrt{f_2'(0)}$ and $\varphi_1$ is the first eigenfunction corresponding to the first Steklov eigenvalue $\mu_1$. 
Indeed, by letting $\zeta_1(s)=\mu_1 \frac{a}{b}s-\lambda f_1(s)$ and $\zeta_2(s)=\mu_1 \frac{b}{a}s-\lambda f_2(s)$, from Hypothesis \eqref{H0}, it clearly follows that $\zeta_1(0)=0$ and $\zeta_2(0)=0$. Moreover, by taking $\lambda > \mu_0=\frac{\mu_1}{ab}$, we get
\begin{align}
    \zeta_1'(0) &= \mu_1\frac{a}{b}-\lambda f_1'(0)\notag < \lambda ab \frac{a}{b}-\lambda a^2 \notag = 0. \notag 
\end{align}
Note that, using the fact that  $\zeta_1(0)=0 $, $\ \zeta_1'(0)<0$ and $\ \zeta_1$ is continuous, it follows that there exists $\delta>0$ such that $\zeta_1(s)<0 \mbox{ for every } 0<s<\delta$. Now, choosing $\e $ small enough such that $0< \e  b\varphi_1 < \delta$, we have that $ \zeta_1(\e  b \varphi_1)= \mu_1  \e  a  \varphi_1- \lambda f_1(\e  b \varphi_1)<0$.
Therefore, for any $\psi\geq 0$
\begin{align}\label{eq:u1}
    \displaystyle \int_{\Omega} \nabla \underline{u}_1 \nabla \psi + \underline{u}_1\psi & 
    = \mu_1 a\e  \displaystyle \int_{\partial \Om}   \varphi_1 \psi  \leq \lambda \displaystyle \int_{\p\Omega} f_1(\e  b \varphi_1) \psi=  \lambda \displaystyle \int_{\p\Omega} f_1(\underline{u}_2) \psi.
\end{align}
Similarly, one can show that for any $\psi\geq 0$
\begin{align}\label{eq:u2}
    \displaystyle \int_{\Omega} \nabla \underline{u}_2 \nabla \psi + \underline{u}_2\psi & = \displaystyle \mu_1 b \e   \int_{\partial \Om}  \varphi_1 \psi
    \leq \lambda \displaystyle \int_{\p\Omega}  f_2( \e  a \varphi_1) \psi= \lambda \displaystyle \int_{\p\Omega} f_2(\underline{u}_1) \psi.
\end{align}
Thus, equations \eqref{eq:u1} and \eqref{eq:u2} together imply that $(\underline{u}_1,\underline{u}_2)=(a\e \varphi_1,b\e \varphi_1)$ is a subsolution to \eqref{pde}. 

Next, we claim that $(\overline{u}_1,\overline{u}_2) = \big(\min\{u_1^*,u_1^0\},\min\{u_2^*,u_2^0\}\big)$ is a strict supersolution to \eqref{pde} in the sense that it is not a solution to \eqref{pde}. Observe that, $(u_1^*,u_2^*) \in Y \setminus \overline{Z} = \{(u_1,u_2) \in Y\colon u_1 \nleq u_1^0 \mbox{ or }u_2 \nleq u_2^0 \}$. 
Additionally, we have
\begin{enumerate}
    \item $(u_1^*,u_2^*)$ is a solution to \eqref{pde}, hence a supersolution.
    \item $(u_1^0,u_2^0)$ is a strict supersolution to \eqref{pde}, since $\lambda< \lambda_0$.
\end{enumerate}
Therefore, by \cite[Lem. 5.5]{BLM_maxmin}, $\big(\min\{u_1^*,u_1^0\},\min\{u_2^*,u_2^0\}\big)$ is a strict supersolution to \eqref{pde}. 

Moreover, we have $a \epsilon \varphi_1 < \min\{u_1^*,u_1^0\}$ and $b \epsilon \varphi_1 <\min \{u_2^*,u_2^0\}$ since both $(u_1^*,u_2^*)$ and $(u_1^0,u_2^0)$ in $Y$. 
This implies $(\underline{u}_1,\underline{u}_2) < (\overline{u}_1, \overline{u}_2)$.
Hence, there exists a solution (see \cite[Thm. 1.4]{BLM_maxmin}) $(\widetilde{u}_1, \widetilde{u}_2)$ of Eq. \eqref{pde}  
such that $(a \epsilon \varphi_1, b \epsilon \varphi_1) \le (\widetilde{u}_1, \widetilde{u}_2) < \big(\min\{u_1^*, u_1^0\},\min\{u_2^*,u_2^0\}\big) \le  (u_1^*,u_2^*)$.

\par Thus we have two distinct solutions $(\widetilde{u}_1, \widetilde{u}_2) < (u_1^*,u_2^*)$ for each $\lambda  \in \left(\mu_0, 
\overline{\lambda}\right)$.
\subsubsection{Step~3: The existence of a solution for $\lambda=\overline{\lambda}$.}
{\it Step 1} and {\it Step 2} imply that for each $\lambda\in \left(\mu_0, \overline{\lambda}\right)$ problem \eqref{pde} admits a positive solution $(u_{1,\lambda}, u_{2,\lambda})$. 
For $\lambda\in \left(\mu_0, \overline{\lambda}\right)$, we have,
\begin{align*}
\|u_{1,\lambda}\|_{H^1(\Omega)}^2=\int_{\Omega} \nabla u_{1,\lambda}^2+\int_{\Omega} u_{1,\lambda}^2=\lambda\int_{\partial \Omega} u_{1,\lambda} \, f_1(u_{2,\lambda})\leq C,\\
\|u_{2,\lambda}\|_{H^1(\Omega)}^2=\int_{\Omega} \nabla u_{2,\lambda}^2+\int_{\Omega} u_{2,\lambda}^2=\lambda\int_{\partial \Omega} u_{2,\lambda} \, f_2(u_{1,\lambda})\leq C,
\end{align*}
and by the reflexivity of the Sobolev space $(H^1(\Omega))^2$, $(u_{1,\lambda}, u_{2,\lambda})$ has a subsequence that converges weakly  to $(u_{1,\overline{\lambda}}, u_{2,\overline{\lambda}})$ in $(H^1(\Omega))^2$ as $\lambda\rightarrow \overline{\lambda}$.
Thereafter, taking the limit in the weak  formulation of $(u_{1,\lambda}, u_{2,\lambda})$ as $\lambda\rightarrow \overline{\lambda}$, we get
\begin{align*}
\int_{\Omega} \nabla u_{1,\overline{\lambda}}\nabla \psi_1 +\int_{\Omega} u_{1,\overline{\lambda}}\psi_1&=\overline{\lambda} \int_{\partial\Omega} f_1(u_{2,\overline{\lambda}})\psi_1\,,\\ 
\int_{\Omega} \nabla u_{2,\overline{\lambda}}\nabla \psi_2 +\int_{\Omega} u_{2,\overline{\lambda}}\psi_2&=\overline{\lambda} \int_{\partial\Omega} f_2(u_{1,\overline{\lambda}})\psi_2.	
\end{align*}
Hence, $(u_{1,\overline{\lambda}}, u_{2,\overline{\lambda}})$ is a weak  positive solution to \eqref{pde}.

Combining Theorem \ref{th:bootstrap}, equation \eqref{eq:bound-norm-u} and Proposition \ref{th:Rossi}, we can find a uniform constant $C>0$ such that
\[
\|(u_{1,\lambda}, u_{2,\lambda})\|_{(C^{\alpha}(\overline{\Omega}))^2}\leq C, \quad \text{ for any }\quad\lambda\in \left(\mu_0, \overline{\lambda}\right).
\]
The compact embeddings of H\"older spaces guarantee that $(u_{1,\lambda}, u_{2,\lambda})$ has a convergent subsequence converging to $(u_{1,\overline{\lambda}}, u_{2,\overline{\lambda}})$ in $(C^{\beta}(\overline{\Omega}))^2$ as $\lambda\rightarrow \overline{\lambda}$ and $\beta<\alpha$. 

Then, since $(u_{1,\lambda}, u_{2,\lambda})\rightarrow (u_{1,\overline{\lambda}}, u_{2,\overline{\lambda}})$ in $(C^{\beta}(\overline{\Omega}))^2$ and $f_i$ are  Hölder continuous,  $f_1(u_{2,\lambda})\rightarrow f_1(u_{2,\overline{\lambda}})$ and $f_2(u_{1,\lambda})\rightarrow f_2(u_{1,\overline{\lambda}})$, respectively, in $C^{\beta'}(\overline{\Omega})$ as $\lambda\rightarrow \overline{\lambda}$.  In fact, the problem \eqref{pde} has at least two positive solutions for $\lambda\in \left(\mu_0, \overline{\lambda}\right)$ and at least one positive solution for $\lambda=\overline{\lambda}$.

\subsubsection{Step~4: The existence of a positive solution for $\lambda=\mu_0$.}
In the final step, note that, since the connected set $\mathscr{C}^+$ bifurcates from trivial solution to the right at $\lambda = \mu_0$, and bifurcates from infinity at $\lambda=0$, we conclude that $\mathscr{C}^+$ must cross the hyperplane $\lambda=\mu_0$ at a point distinct to zero. Consequently, the problem \eqref{pde} has a positive weak solution at $\lambda=\mu_0$, as we desired.

\par This completes the proof of the theorem.
\qed

\section*{Acknowledgements}
This material is based upon work supported by the National Science Foundation under Grant No. 1440140, while the authors were in residence at the Mathematical Sciences Research Institute in Berkeley, California, during the month of June of 2022. Chhetri was supported by Simons Foundation Grant No.~965180, R. Pardo was supported by PID2022-137074NB-I00,  MICINN,  Spain, and by UCM, Spain,  Grupo 920894. Bandyopadhyay was supported by the AMS-Simons Travel Award.
N. Mavinga was supported by AMS-Simons Research Enhancement Grants for Primarily Undergraduate Institution (PUI) Faculty, 2024 - 2027.\\


\noindent {\bf Author Contributions}

S.B., M.C., B.D., N.M. and R.P. wrote the main manuscript text. All authors reviewed the manuscript.\\

\noindent {\bf Data availability} 

No datasets were generated or analyzed during the current study.\\

\section*{Declarations}
{\bf Conflict of interest}: The authors declare no conflict of interest.
    
\bigskip
\bibliographystyle{abbrv}
\bibliography{references}

\end{document}